\documentclass[12pt, reqno]{article}
\usepackage{amsfonts}
\usepackage{amsmath}
\usepackage{amssymb}
\usepackage{latexsym,bm}
\usepackage{amsthm}
\usepackage{graphicx}
\usepackage{cite}
\usepackage[colorlinks,urlcolor=blue]{hyperref}

\usepackage[top=3cm,bottom=3cm,left=3cm,right=3cm,includehead,includefoot]{geometry}
\usepackage{algorithm}
\usepackage{algorithmic}
\allowdisplaybreaks
\usepackage{multirow}
\usepackage{threeparttable}
\usepackage{slashbox}

\newtheorem{thm}{Theorem}[section]

\newtheorem{lem}{Lemma}[section]
\newtheorem{prop}[thm]{Proposition}
\newtheorem{ass}{Assumption}[section]
\theoremstyle{definition}
\newtheorem{defn}{Definition}[section]
\theoremstyle{Condition}

\theoremstyle{remark}

\numberwithin{equation}{section}
\theoremstyle{example}

\numberwithin{equation}{section}

\begin{document}

\bigskip
\bigskip

\bigskip

\begin{center}

\textbf{\large A stochastic coordinate descent primal-dual algorithm
with dynamic stepsize for large-scale composite optimization }

\end{center}

\begin{center}
Meng Wen $^{1,2}$, Shigang Yue$^{4}$, Yuchao Tang$^{3}$, Jigen Peng$^{1,2}$
\end{center}

\begin{center}
1. School of Mathematics and Statistics, Xi'an Jiaotong University,
Xi'an 710049, P.R. China \\
2. Beijing Center for Mathematics and Information Interdisciplinary
Sciences, Beijing, P.R. China

 3. Department of Mathematics, NanChang University, Nanchang
330031, P.R. China\\
4. School of Computer Science, University of Lincoln, LN6 7TS, UK
\end{center}

\footnotetext{\hspace{-6mm}$^*$ Corresponding author.\\
E-mail address: wen5495688@163.com}

\bigskip

\noindent  \textbf{Abstract} In this paper we consider the problem
of finding the minimizations of the sum  of two convex functions and
the composition of another convex function with a continuous linear
operator. With the idea of coordinate descent, we design a
stochastic coordinate descent primal-dual splitting algorithm
 with dynamic stepsize. Based on randomized Modified Krasnosel'skii-Mann iterations and the firmly
nonexpansive properties of the proximity operator, we achieve the
convergence of the proposed  algorithms. Moreover, we give two
applications of our method. (1) In the case of stochastic minibatch
optimization, the algorithm can be applicated to split a composite
objective function into blocks, each of these blocks being processed
sequentially by the computer. (2) In the case of distributed
optimization, we consider a set of $N$ networked agents endowed with
private cost functions and seeking to find a consensus on the
minimizer of the aggregate cost. In that case, we obtain a
distributed iterative algorithm where isolated components of the
network are activated in an uncoordinated fashion and passing in an
asynchronous manner.  Finally, we illustrate the efficiency of the
method in the framework of large scale machine learning
applications. Generally speaking, our method  is comparable with
other state-of-the-art methods in numerical performance, while it
has some advantages on parameter selection in real applications.


\noindent \textbf{MR(2000) Subject Classification} 47H09, 90C25,

\section{Introduction}

The purpose of this paper is to designing and discussing an
efficient algorithmic framework with dynamic stepsize for minimizing
the following problem
$$\min_{x\in\mathcal{X}} f(x)+g(x)+ (h\circ D)(x),\eqno{(1.1)}$$
where $\mathcal{X}$ and $\mathcal{Y}$ are two Euclidean spaces,
$f,g\in\Gamma_{0}(\mathcal{X})$, $h\in\Gamma_{0}(\mathcal{Y}),$ and
$f$ is differentiable on $\mathcal{Y}$ with a $\beta$-Lipschitz
continuous gradient for some $\beta\in(0,+\infty)$ and
$D:\mathcal{X}\rightarrow\mathcal{Y}$ a linear transform. This
parameter $\beta$ is related to the convergence conditions of
algorithms presented in the following section. Here and in what
follows, for a real Hilbert space $\mathcal{H}$,
$\Gamma_{0}(\mathcal{H})$ denotes the collection of all proper lower
semi-continuous convex functions from $\mathcal{H}$ to
$(-\infty,+\infty]$. Despite its simplicity, when $g=0$ many
problems in image processing can be formulated in the form of (1.1).

\par
In this paper, the contributions of us are the following aspects:
\par
(I)we  provide a more general iteration in which the coefficient
$\tau$, $\sigma$ is made iteration-dependent to solve the general
Problem (1.1), errors are allowed in the evaluation of the operators
$prox_{\sigma h^{\ast}}$and $prox_{\tau g}$. The errors allow for
some tolerance in the numerical implementation of the algorithm,
while the flexibility introduced by the iteration-dependent
parameters $\tau_{k}$ and  $\sigma_{k}$  can be used to improve its
convergence pattern. We refer to our algorithm as ADMMDS$^{+}$, and
when $\tau_{k}\equiv\tau$, $\sigma_{k}\equiv\sigma$, the ADMM$^{+}$
algorithm introduced by Bianchi [2] is a special case of our
algorithm.

\par
(II) Based on the results of Bianchi [2], we introduce the idea of
stochastic coordinate descent on  modified krasnosel¡¯skii mann
iterations. The form of Modified Krasnosel'skii-Mann iterations can
be translated into fixed point iterations of a given operator having
a contraction-like property. Interestingly, ADMMDS$^{+}$ is a
special instances of Modified Krasnosel'skii-Mann iterations. By the
view of stochastic coordinate descent, we know that at each
iteration, the algorithm is only to update  a random subset of
coordinates. Although this leads to a perturbed version of the
initial Modified Krasnosel'skii-Mann iterations, but it can be
proved to preserve the convergence properties of the initial
unperturbed version. Moreover, stochastic coordinate descent has
been used in the literature [18-20] for proximal gradient
algorithms. We believe that its application to  the broader class of
Modified Krasnosel'skii-Mann algorithms can potentially lead to
various algorithms well suited to large-scale optimization problems.

\par
(III) We use our views to large-scale optimization problems which
arises in signal processing and machine learning contexts. We prove
that the general idea of stochastic coordinate descent gives a
unified framework allowing to derive stochastic algorithms with
dynamic stepsize of different kinds. Furthermore, we give two
application examples. Firstly, we propose a new stochastic
approximation algorithm with dynamic stepsize by applying stochastic
coordinate descent on the top of  ADMMDS$^{+}$. The algorithm is
called as stochastic minibatch  primal-dual splitting algorithm with
dynamic stepsize (SMPDSDS). Secondly, we introduce a random
asynchronous distributed optimization methods with dynamic stepsize
that we call as distributed asynchronous primal-dual splitting
algorithm with dynamic stepsize (DAPDSDS). The algorithm can be used
to efficiently solve an optimization problem over a network of
communicating agents.  The algorithms are asynchronous in the sense
that some components of the network are allowed to wake up at random
and perform local updates, while the rest of the network stands
still. No coordinator or global clock is needed. The frequency of
activation of the various network components is likely to vary.

The rest of this paper is organized as follows. In the next section,
 we introduce some notations used throughout in the paper. In section 3, we devote
to introduce  PDSDS and ADMMDS$^{+}$ algorithm,  and the relation
between them, we also show how the ADMMDS$^{+}$ includes ADMM$^{+}$
and the Forward-Backward algorithm as  special cases. In section 4,
we provide our main result on the convergence of Modified
Krasnosel'skii-Mann algorithms with randomized coordinate descent.
 In section 5, we propose a stochastic approximation
algorithm from the ADMMDS$^{+}$.  In section 6, we addresse the
problem of asynchronous distributed optimization. In the final
section, we show the numerical performance and efficiency of propose
algorithm through some examples in the context of large-scale
$l_{1}$-regularized logistic regression.

\section{Preliminaries }
Throughout the paper, we  denote by $\langle \cdot, \cdot\rangle$
the inner product on $\mathcal{X}$ and by $\|\cdot\|$ the norm on
$\mathcal{X}$.

\begin{ass}
The infimum of Problem (1.1) is attained. Moreover, the following
qualification condition holds
$$0\in ri(dom\, h-D\, dom\, g).$$
\end{ass}
The dual problem corresponding to the primal Problem (1.1) is
written
$$\min_{y\in\mathcal{Y}} (f+g)^{\ast}(-D^{\ast}y)+ h^{\ast}(y),$$
where $a^{\ast}$ denotes the Legendre-Fenchel transform of a
function $a$ and where $D^{\ast}$ is the adjoint of $D$. With the
Assumption 2.1, the classical Fenchel-Rockafellar duality theory
[3], [10] shows that
$$\min_{x\in\mathcal{X}} f(x)+g(x)+ (h\circ D)(x)-\min_{y\in\mathcal{Y}} (f+g)^{\ast}(-D^{\ast}y)+ h^{\ast}(y).\eqno{(2.1)}$$

\begin{defn}
 Let $f$ be  a real-valued convex function on
$\mathcal{X}$, the operator prox$_{f}$ is defined by
\begin{align*}
prox_{f}&:\mathcal{X}\rightarrow\mathcal{X}\\
& x\mapsto \arg \min_{y\in
\mathcal{X}}f(y)+\frac{1}{2}\|x-y\|_{2}^{2},
\end{align*}
called the proximity operator of $f$.

\end{defn}

\begin{defn}
Let $A$ be a closed convex set of $\mathcal{X}$. Then the indicator
function of $A$ is defined as
$$
\iota_{A}(x) = \left\{
\begin{array}{l}
0,\,\,\, \,\,if x\in A,\\
\infty,\,\,\, otherwise .
\end{array}
\right.
$$
\end{defn}
It can easy see the proximity operator of the indicator function in
a closed convex subset $A$ can be reduced a projection operator onto
this closed convex set $A$. That is,
$$prox_{\iota_{A}}=proj_{A}$$

where proj is the projection operator of $A$.

\begin{defn}
(Nonexpansive operators and firmly nonexpansive operators [3]). Let
$\mathcal{H}$ be a Euclidean space (we refer to [3] for an extension
to Hilbert spaces). An operator $T : \mathcal{H} \rightarrow
{\mathcal{H}}$ is nonexpansive if and only if it satisfies
$$\|Tx-Ty\|_{2}\leq\|x-y\|_{2}\,\,\, for\,\,all\,\,\, (x,y)\in \mathcal{H}^{2}.$$
$T$ is firmly nonexpansive if and only if it satisfies one of the
following equivalent conditions:
\par
(i)$\|Tx-Ty\|_{2}^{2}\leq\langle Tx-Ty,x-y\rangle$\,\,\,
for\,\,all\,\,\, $(x,y)\in \mathcal{H}^{2}$;
\par
(ii)$\|Tx-Ty\|_{2}^{2}=\|x-y\|_{2}^{2}-\|(I-T)x-(I-T)y\|_{2}^{2}$\,\,\,
for\,\,all\,\,\, $(x,y)\in \mathcal{H}^{2}$.
\par
It is easy to show from the above definitions that a firmly
nonexpansive operator $T$ is nonexpansive.
\end{defn}
\begin{defn}
 A mapping $T : \mathcal{H} \rightarrow \mathcal{H}$ is said to be an averaged mapping, iff it can
be written as the average of the identity $I$ and a nonexpansive
mapping; that is,
$$T = (1-\alpha)I +\alpha S,\eqno{(2.2)}$$
where $\alpha$ is a number in ]0, 1[ and $S : \mathcal{H}
\rightarrow \mathcal{H}$ is nonexpansive. More precisely, when (2.2)
or the following inequality (2.3) holds, we say that $T$ is
$\alpha$-averaged.
$$\|Tx-Ty\|^{2}\leq \|x-y\|^{2}-\frac{(1-\alpha)}{\alpha}\|(I -T)x-(I -T)y\|^{2},\forall x,y\in\mathcal{H}.\eqno{(2.3)}$$

\end{defn}
A 1-averaged operator is said non-expansive. A $\frac{1}{ 2}$
-averaged operator is said firmly non-expansive.

 We refer the
readers to [3] for more details. Let $M : \mathcal{H} \rightarrow
\mathcal{H}$ be a set-valued operator. We denote by $ran(M) := \{v
\in\mathcal{H} : \exists u \in\mathcal{H}, v \in Mu\}$ the range of
$M$, by $ gra(M) := {(u, v) \in \mathcal{H}^{2} : v \in Mu}$ its
graph, and by $M ^{-1}$ its inverse; that is, the set-valued
operator with graph ${(v, u) \in \mathcal{H}^{2} : v \in Mu}$. We
define $ zer(M) := {u \in \mathcal{H} : 0\in Mu}$. $M$ is said to be
monotone iff $\forall(u, u' ) \in \mathcal{H}^{2},\forall(v, v' )
\in Mu\times Mu'$, $\langle u-u' , v-v' \rangle\geq 0$ and maximally
monotone iff there exists no monotone operator $M'$ such that $
gra(M) \subset gra(M) \neq gra(M)$.

The resolvent $(I + M)^{-1}$ of a maximally monotone operator $M :
\mathcal{H} \rightarrow \mathcal{H}$ is defined and single-valued on
$\mathcal{H}$ and firmly nonexpansive. The subdifferential $\partial
J$ of $J\in \in\Gamma_{0}(\mathcal{H})$ is maximally monotone and
$(I +\partial J)^{-1} = prox_{J}$ .

\begin{lem}
(Krasnosel'skii-Mann iterations [3]) Assume that $T :
\mathcal{H}\rightarrow \mathcal{H}$ is $\frac{1}{\delta}$-averaged
and that the set $Fix(T)$ of fixed points of $T$ is non-empty.
Consider a sequence $(\rho_{k})_{k\in\mathbb{N}}$ such that $0\leq
\rho_{k} \leq \delta$ and $\sum_{k} \rho_{k}(\delta -\rho_{k}) =
\infty$. For any $x^{0} \in \mathcal{H}$, the sequence
$(x^{k})_{k\in\mathbb{N}}$ recursively defined on $\mathcal{H}$ by $
x^{k+1} = x^{k} + \rho_{k}(Tx^{k}- x^{k})$ converges to some point
in $Fix(T)$.
\end{lem}
\begin{lem}
(Baillon-Haddad Theorem [3, Corollary 18.16]). Let $J : \mathcal {H}
\rightarrow \mathcal {R}$ be convex, differentiable on $\mathcal
{H}$ and such that $\pi\nabla J$ is nonexpansive, for some $\pi\in
]0,+\infty[$. Then $\nabla J$ is $\pi$-cocoercive; that
is,$\pi\nabla J$ is firmly nonexpansive.
\end{lem}
\begin{lem}
((Composition of averaged operators [4, Theorem 3]). Let
$\alpha_{1}\in ]0, 1[$, $\alpha_{2}\in  ]0, 1]$, $T_{1}\in\mathcal
{A}(\mathcal {H},\alpha_{1})$, and $T_{2}\in\mathcal {A}(\mathcal
{H},\alpha_{2})$. Then $T_{1}\circ T_{2}\in \mathcal {A}(\mathcal
{H},\alpha')$,
 where
$$\alpha':=\frac{\alpha_{1}+\alpha_{2}-2\alpha_{1}\alpha_{2}}{1-\alpha_{1}\alpha_{2}}.$$
\end{lem}
\begin{prop}
([5,6]). Let  $\tilde{H}$ be a Hilbert space, and the operators $T :
\tilde{H}\rightarrow\tilde{H}$ be given. If the mappings
$\{T_{i}\}^{N} _{i=1}$ are averaged and have a common fixed point,
then
$$\bigcap_{i=1}^{N} Fix(T_{i})=Fix(T_{1}\cdots T_{N}).$$
Here the notation $Fix(T ) \equiv Fix T$ denotes the set of fixed
points of the mapping $T$ ; that is, $Fix T := \{x \in \tilde{H} : T
x = x\}$.
\end{prop}
Averaged mappings are useful in the convergence analysis, due to the following
result.
\begin{prop}
([7]). Let $T : \tilde{H}\rightarrow\tilde{H}$ an averaged mapping.
Assume that $T$ has
a bounded orbit, i.e., $\{T ^{k}x^{0}\}^{\infty}_{k=0}$ is bounded for some $x^{0} \in \tilde{H}$. Then we have:\\
(i) $T$ is asymptotically regular, that is, $\lim_{k\rightarrow\infty}\|T^{k+1}x-T ^{k}x\|=0$, for all $x \in \tilde{H}$;\\
(ii) for any $x \in \tilde{H}$, the sequence $\{T
^{k}x\}^{\infty}_{k=0}$ converges  to a fixed point of $T$.
\end{prop}
The so-called demiclosedness principle for nonexpansive mappings will often be
used.
\begin{lem}
((Demiclosedness Principle [7]). Let $C$ be a closed and convex
subset of a Hilbert space $\tilde{H}$ and let $T : C\rightarrow C$
be a nonexpansive mapping with $Fix T\neq\emptyset $. If $
\{x^{k}\}^{\infty}_{k=1}$ is a sequence in $C$ weakly converging to
$x$ and if $\{(I-T )x^{k}\}^{\infty}_{k=1}$ converges strongly to
$y$, then $(I-T )x = y$. In particular, if $y = 0$, then $x \in Fix
T$.
\end{lem}
\begin{lem}
(The Resolvent Identity [8,9]). For $\lambda>0$ and $\nu>0$ and $x\in \tilde{E}$, where $\tilde{E}$ is a Banach sapce, \\
$$J_{\lambda}x=J_{\nu}(\frac{\nu}{\lambda}+(1-\frac{\nu}{\lambda})J_{\lambda}x).$$
\end{lem}

\section{A primal-dual splitting algorithm
 with dynamic stepsize }
\subsection{Derivation of the algorithm}

For Problem (1.1), Condat [1] considered a primal-dual splitting
method as follows:

$$
\left\{
\begin{array}{l}
\tilde{y}^{k+1}=prox_{\sigma h^{\ast}}(y^{k}+\sigma Dx^{k}),\\
\tilde{x}^{k+1}=prox_{\tau g}(x^{k}-\tau \nabla f(x^{k})-\tau
D^{\ast}(2\tilde{y}^{k+1}-y^{k})),\\
(x^{k+1}, y^{k+1})=\rho_{k}(\tilde{x}^{k+1},
\tilde{y}^{k+1})+(1-\rho_{k})(x^{k}, y^{k})
\end{array}
\right.\eqno{(3.1)}
$$
Then, the corresponding algorithm is given below, called Algorithm
1.
\begin{algorithm}[H]
\caption{A primal-dual splitting algorithm(PDS).}
\begin{algorithmic}\label{1}
\STATE Initialization: Choose $x^{0}\in \mathcal{X}$, $ y^{0}\in
\mathcal{Y}$, relaxation parameters $(\rho_{k})_{k\in \mathbb{N}}$, and proximal\\
~~~~~~~~~~~~~~~~~~ parameters $\sigma>0$, $\tau>0$.\\
Iterations ($k\geq0$): Update $x^{k}$, $y^{k}$ as follows
$$
\left\{
\begin{array}{l}
\tilde{y}^{k+1}=prox_{\sigma h^{\ast}}(y^{k}+\sigma Dx^{k}),\\
\tilde{x}^{k+1}=prox_{\tau g}(x^{k}-\tau \nabla f(x^{k})-\tau
D^{\ast}(2\tilde{y}^{k+1}-y^{k})),\\
(x^{k+1}, y^{k+1})=\rho_{k}(\tilde{x}^{k+1},
\tilde{y}^{k+1})+(1-\rho_{k})(x^{k}, y^{k}).
\end{array}
\right.
$$
end for
\end{algorithmic}
\end{algorithm}
For Algorithm 1, the author given the following Theorem.

\begin{thm}([1])
Let $\sigma>0$, $\tau>0$ and the sequences $(\rho_{k})_{k\in
\mathbb{N}}$, be the parameters of Algorithms 1. Let $\beta$ be the
Lipschitz constant and suppose that $\beta>0$. Then the following
hold:\\
(i) $\frac{1}{\tau}-\sigma\|D\|^{2}>0$,\\
(ii) $\forall k\in \mathbb{N}$, $\rho_{k}\in]0,\delta[$, where
$\delta=2-\frac{\beta}{2}(\frac{1}{\tau}-\sigma\|D\|^{2})^{-1}\in[1,2[$,\\
(iii) $\sum_{k\in\mathbb{N}}\rho_{k}(\delta-\rho_{k})=+\infty$.\\
Let the sequences $(x^{k},y^{k})$ be generated by Algorithms 1. Then
the sequence $\{x_{k}\}$ converges to a solution of Problem (1.1).
\end{thm}
The fixed point characterization provided by Condat [1] suggests
solving Problem (1.1 ) via the fixed point iteration scheme (3.1)
for a suitable value of the parameter $\sigma>0$, $\tau>0$. This
iteration, which is referred to as a primal-dual splitting method
for convex optimization involving Lipschitzian, proximable and
linear composite terms. A very natural idea is to provide a more
general iteration in which the coeffiient $\sigma>0$ and $\tau>0$
are made iteration-dependent to solve the general Problem (1.1),
then we can obtain the following iteration scheme:
$$
\left\{
\begin{array}{l}
\tilde{y}^{k+1}=prox_{\sigma_{k} h^{\ast}}(y^{k}+\sigma_{k} Dx^{k}),\\
\tilde{x}^{k+1}=prox_{\tau_{k} g}(x^{k}-\tau_{k} \nabla
f(x^{k})-\tau_{k}
D^{\ast}(2\tilde{y}^{k+1}-y^{k})),\\
(x^{k+1}, y^{k+1})=\rho_{k}(\tilde{x}^{k+1},
\tilde{y}^{k+1})+(1-\rho_{k})(x^{k}, y^{k})
\end{array}
\right.\eqno{(3.2)}
$$
which produces our proposed method Algorithm 3.2, described below.
This algorithm can also be deduced from the fixed point formulation,
whose detail we will give in the following. On the other hand, since
the parameter $\sigma_{k}>0$ and $\tau_{k}>0$ are dynamic, so we
call our method a primal-dual splitting algorithm with dynamic
stepsize, and abbreviate it as PDSDS. If $\sigma_{k}\equiv\sigma$
and $\tau_{k}\equiv\tau$ then form (3.1) is equivalent to form
(3.2). So PDS can be seen as a special case of PDSDS.

\begin{algorithm}[H]
\caption{A primal-dual splitting algorithm with dynamic
stepsize(PDSDS).}
\begin{algorithmic}\label{1}
\STATE Initialization: Choose $x^{0}\in \mathcal{X}$, $ y^{0}\in
\mathcal{Y}$, relaxation parameters $(\rho_{k})_{k\in \mathbb{N}}$, and proximal\\
~~~~~~~~~~~~~~~~~~ parameters $\liminf_{k\rightarrow\infty}\sigma_{k}>0$,
, $\liminf_{k\rightarrow\infty}\tau_{k}>0$.\\
Iterations ($k\geq0$): Update $x^{k}$, $y^{k}$ as follows
$$
\left\{
\begin{array}{l}
\tilde{y}^{k+1}=prox_{\sigma_{k} h^{\ast}}(y^{k}+\sigma_{k} Dx^{k}),\\
\tilde{x}^{k+1}=prox_{\tau_{k} g}(x^{k}-\tau_{k} \nabla
f(x^{k})-\tau_{k}
D^{\ast}(2\tilde{y}^{k+1}-y^{k})),\\
(x^{k+1}, y^{k+1})=\rho_{k}(\tilde{x}^{k+1},
\tilde{y}^{k+1})+(1-\rho_{k})(x^{k}, y^{k})
\end{array}
\right.
$$
end for
\end{algorithmic}
\end{algorithm}
Now, we claim  the convergence results for Algorithms 2.
\begin{thm}
Assume that the minimization Problem (1.1) is consistent,
$\liminf_{k\rightarrow\infty}\sigma_{k}>0$, and
$\liminf_{k\rightarrow\infty}\tau_{k}>0$.  Let the sequences
$(\rho_{k})_{k\in \mathbb{N}}$, be the parameters of Algorithms 2.
Let $\beta$ be the Lipschitz constant and suppose that $\beta>0$.
Then the following
hold:\\
(i) $\frac{1}{\liminf_{k\rightarrow\infty}\tau_{k}}-\liminf_{k\rightarrow\infty}\sigma_{k}\|D\|^{2}>\frac{\beta}{2}$,\\
(ii) $\forall k\in \mathbb{N}$, $\rho_{k}\in]0,\delta_{k}[$, where
$\delta_{k}=2-\frac{\beta}{2}(\frac{1}{\tau_{k}}-\sigma_{k}\|D\|^{2})^{-1}\in[1,2[$,\\
(iii)$0<\liminf_{k\rightarrow\infty}\rho_{k}\leq\limsup_{k\rightarrow\infty}\rho_{k}<\limsup_{k\rightarrow\infty}\delta_{k}$ and $1\leq\liminf_{k\rightarrow\infty}\delta_{k}\leq\limsup_{k\rightarrow\infty}\delta_{k}<2$. \\
Let the sequences $(x^{k},y^{k})$ be generated by Algorithms 2. Then
the sequence $\{x_{k}\}$ converges to a solution of Problem (1.1).
\end{thm}
We consider the case where $D$ is injective(in
particular, it is implicit that dim$(\mathcal{X})\leq$
dim$(\mathcal{Y}))$. In the latter case, we denote by $\mathcal{R}$
= Im$(D)$ the image of $D$ and by $D^{-1}$ the inverse of $D$ on
$\mathcal{R}\rightarrow\mathcal{X}$. We emphasize the fact that the
inclusion $\mathcal{R}\subset\mathcal{Y}$ might be strict. We denote
by $\nabla$ the gradient operator. We make the following
assumption:

\begin{ass}
The following facts holds true:\\
 (1)$D$ is injective;\\
 (2)$\nabla(f\circ D)^{-1}$  is L-Lipschitz continuous on $\mathcal {R}$.
\end{ass}

For proximal parameters $\liminf_{k\rightarrow\infty}\mu_{k}>0$,
 $\liminf_{k\rightarrow\infty}\tau_{k}>0$, we consider the following algorithm which we shall refer to as
ADMMDS$^{+}$.
\begin{algorithm}[H]
\caption{ADMMDS$^{+}$.}
\begin{algorithmic}\label{1}
\STATE Iterations ($k\geq0$): Update $x^{k}$, $u^{k}$, $y^{k}$,
$z^{k}$ as follows
$$
\left\{
\begin{array}{l}
z^{k+1}=\arg\min_{w\in\mathcal {Y}}[h(w)+\frac{\|w-(Dx^{k}+\mu_{k}y^{k})\|^{2}}{2\mu_{k}}],~~~~~~~~~~~~~~~~~~~~~~~~~~~~~~~~~~~~~~~~~~~~~(a)\\
y^{k+1}=y^{k}+\mu_{k}^{-1}(Dx^{k}-z^{k+1}),~~~~~~~~~~~~~~~~~~~~~~~~~~~~~~~~~~~~~~~~~~~~~~~~~~~~~~~~~~~~~~(b)\\
u^{k+1}=(1-\tau_{k}\mu_{k}^{-1})Dx^{k}+\tau_{k}\mu_{k}^{-1}z^{k+1},~~~~~~~~~~~~~~~~~~~~~~~~~~~~~~~~~~~~~~~~~~~~~~~~~~~~~(c)\\
x^{k+1}=\arg\min_{w\in\mathcal {X}}[g(w)+\langle\nabla
f(x^{k}),w\rangle+\frac{\|Dw-u^{k+1}-\tau_{k}y^{k+1}\|^{2}}{2\tau_{k}}]~~~~~~~~~~~~~~~~~~~~~~~~(d)
\end{array}
\right.
$$
end for
\end{algorithmic}
\end{algorithm}
\begin{thm}
Assume that the minimization Problem (1.1) is consistent,
$\liminf_{k\rightarrow\infty}\mu_{k}>0$, and
$\liminf_{k\rightarrow\infty}\tau_{k}>0$.  Let Assumption 2.1 and
Assumption 3.1 hold true and
$\frac{1}{\liminf_{k\rightarrow\infty}\tau_{k}}-\frac{1}{\liminf_{k\rightarrow\infty}\mu_{k}}>\frac{L}{2}$.
Let the sequences $(x^{k},y^{k})$ be generated by Algorithms 3. Then
the sequence $\{x_{k}\}$ converges to a solution of Problem (1.1).
\end{thm}

\subsection{Proofs of convergence}
From the proof of Theorem 3.1 for Algorithm 1, we know that
Algorithm 1 has the structure of a forward-backward iteration, when
expressed in terms of nonexpansive operators on $\mathcal {Z} :=
\mathcal {X}\times\mathcal {Y}$, equipped with a particular inner
product.
\par
Let the inner product $\langle\cdot,\cdot\rangle_{I}$ in $\mathcal
{Z}$ be defined as
$$\langle z,z'\rangle:=\langle x,x'\rangle+\langle y,y'\rangle,~~~~~\forall z=(x,y),~ z'=(x',y')\in\mathcal {Z}.\eqno{(3.3)}$$
By endowing $\mathcal {Z}$ with this inner product, we obtain the
Euclidean space denoted by $\mathcal {Z}_{I}$ . Let us define the
bounded linear operator on $\mathcal {Z}$,

$$
P:= \left(
  \begin{array}{ccccccc}
    x  \\
    y  \\

  \end{array}
\right)\mapsto\left(
  \begin{array}{ccccccc}
    \frac{1}{\tau} & -D^{\ast}\\
    -D & \frac{1}{\sigma}\\

  \end{array}
\right)\left(
  \begin{array}{ccccccc}
    x  \\
    y  \\

  \end{array}
\right).\eqno{(3.4)}
$$
From the assumptions $\beta >0$ and (i), we can easily check that
$P$ is positive definite. Hence, we can define another inner product
$\langle\cdot,\cdot\rangle_{P}$ and norm
$\|\cdot\|_{P}=\langle\cdot,\cdot\rangle_{P}^{\frac{1}{2}}$ in
$\mathcal {Z}$ as
$$\langle z,z'\rangle_{P}=\langle z,z'\rangle_{I}.\eqno{(3.5)}$$
We denote by $\mathcal {Z}_{P}$ the corresponding Euclidean space.
\begin{lem}
( [1]). Let the conditions (i)-(iv) in Theorem 3.1 be ture . For
every $n\in \mathbb{N}$, the following inclusion is satisfied by
$\tilde{z}^{k+1} := (\tilde{x}^{k+1}, \tilde{y}^{k+1})$ computed by
Algorithm 1:
$$\tilde{z}^{k+1}:=(I+P^{-1}\circ A)^{-1}\circ(I-P^{-1}\circ B)(z^{k}),\eqno{(3.6)}$$
where $$ A:= \left(
  \begin{array}{ccccccc}
    \partial g & D^{\ast} \\
    -D & \partial h^{\ast} \\

  \end{array}
\right),
 B:= \left(
  \begin{array}{ccccccc}
    \nabla f \\
   0 \\

  \end{array}
\right).
$$
Set $M_{1}=P^{-1}\circ A$, $M_{2}=P^{-1}\circ B$,
$T_{1}=(I+M_{1})^{-1}$, $T_{2}=(I-M_{2})^{-1}$, and $T=T_{1}\circ
T_{2}$. Then $T_{1}\in\mathcal {A}(\mathcal {Z}_{P},\frac{1}{2})$
and $T_{2}\in\mathcal {A}(\mathcal {Z}_{P},\frac{1}{2\kappa})$,
$\kappa:=(\frac{1}{\tau}-\sigma\|D\|^{2})/\beta$. Then $T\in\mathcal
{A}(\mathcal {Z}_{P},\frac{1}{\delta})$ and
$\delta=2-\frac{1}{2\kappa}$.

\end{lem}
In association with  Lemma 2.1 and Lemma 3.1, we obtained Theorem
3.1
\par
Now, we are ready to prove Theorem 3.2

\begin{proof}
By setting $$ P_{k}:= \left(
  \begin{array}{ccccccc}
    \frac{1}{\tau_{k}} & -D^{\ast}\\
    -D & \frac{1}{\sigma_{k}}\\

  \end{array}
\right),
$$
then the Algorithm 3.2 can be described as  follows:
$$\tilde{z}^{k+1}:=(I+P^{-1}_{k}\circ A)^{-1}\circ(I-P^{-1}_{k}\circ B)(z^{k}).\eqno{(3.7)}$$
Considering  the relaxation step, we obtain
$$z^{k+1}:=\rho_{k}(I+P^{-1}_{k}\circ A)^{-1}\circ(I-P^{-1}_{k}\circ B)(z^{k})+(1-\rho_{k})z^{k}.\eqno{(3.8)}$$
Let $M_{1}^{k}=P^{-1}_{k}\circ A$, $M_{2}^{k}=P^{-1}_{k}\circ B$,
$T_{1}^{k}=(I+M_{1}^{k})^{-1}$, $T_{2}^{k}=(I-M_{2}^{k})^{-1}$, and
$T^{k}=T_{1}^{k}\circ T_{2}^{k}$. Then $T_{1}^{k}\in\mathcal
{A}(\mathcal {Z}_{P},\frac{1}{2})$[3, Corollary 23.8].

 First, let us prove the
cocoercivity of  $M_{2}^{k} $. Since the sequence $\tau_{k}$ is
bounded, there exists a convergent subsequence converges to $\tau$
without loss of generality, we may assume that the convergent
subsequence is $\tau_{k}$ itself, then we have
$\frac{1}{\tau_{k}}\rightarrow\frac{1}{\tau}$, so $\forall
\varepsilon>0$, $\exists N_{1}$, such that when $n\geq N_{1}$,
$\frac{1}{\tau_{k}}\geq\frac{1}{\tau}-\varepsilon$. With the same idea, for sequence $\sigma_{k}$, we aslo have $\sigma_{k}\rightarrow\sigma$, then for the above $\varepsilon$
$\exists N_{2}$, such that when $n\geq N_{2}$, $\sigma_{k}\leq\sigma+\varepsilon$. Set $N_{0}=\max\{N_{1},N_{2}\}$, when  $n\geq N_{0}$, we have $\frac{1}{\tau_{k}}\geq\frac{1}{\tau}-\varepsilon$
, $\sigma_{k}\leq\sigma+\varepsilon$.
  Then for every
$z=(x,y), z'=(x',y')\in\mathcal {Z}$ and $\forall n\geq N_{0}$, we have
\begin{align*}
\|M_{2}^{k}(z)-M_{2}^{k}(z')\|_{P}^{2}&=\frac{1}{(\frac{1}{\tau_{k}}-\sigma_{k}DD^{\ast})}\|\nabla f(x)-\nabla f(x')\|^{2}\\
&+\frac{1}{(\frac{1}{\tau_{k}}-\sigma_{k}DD^{\ast})^{2}}(\frac{1}{\tau}-\frac{1}{\tau}_{k})
\|\nabla f(x)-\nabla f(x')\|^{2}\\
&+\frac{\sigma_{k}D^{2}}{(\frac{1}{\tau_{k}}-\sigma_{k}DD^{\ast})^{2}}(\frac{\sigma_{k}}{\sigma}-1)\|\nabla f(x)-\nabla f(x')\|^{2}\\
&\leq\frac{1}{(\frac{1}{\tau_{k}}-\sigma_{k}\|D\|^{2})}\|\nabla f(x)-\nabla f(x')\|^{2}\\
&+\frac{\varepsilon}{(\frac{1}{\tau_{k}}-\sigma_{k}DD^{\ast})^{2}}
\|\nabla f(x)-\nabla f(x')\|^{2}\\
&+\frac{\sigma_{k}D^{2}}{(\frac{1}{\tau_{k}}-\sigma_{k}DD^{\ast})^{2}}\frac{\varepsilon}{\sigma}\|\nabla f(x)-\nabla f(x')\|^{2}\\
&=\frac{1}{(\frac{1}{\tau_{k}}-\sigma_{k}DD^{\ast})}\|\nabla f(x)-\nabla f(x')\|^{2}\\
&+\frac{\varepsilon}{(\frac{1}{\tau_{k}}-\sigma_{k}DD^{\ast})^{2}}(1+\frac{\sigma_{k}}{\sigma}D^{2})
\|\nabla f(x)-\nabla f(x')\|^{2},
\end{align*}
by the arbitrariness of $\varepsilon$, we have
\begin{align*}
\|M_{2}^{k}(z)-M_{2}^{k}(z')\|_{P}^{2}&\leq\frac{1}{(\frac{1}{\tau_{k}}-\sigma_{k}\|D\|^{2})}\|\nabla f(x)-\nabla f(x')\|^{2}\\
&=\frac{1}{\pi_{k}\beta}\|\nabla f(x)-\nabla f(x')\|^{2}\\
&\leq\frac{\beta}{\pi_{k}}\|x-x'\|^{2},\tag{3.9}
\end{align*}
where $\pi_{k}=(\frac{1}{\tau_{k}}-\sigma_{k}\|D\|^{2})/\beta$.
We define the linear operator $Q : (x, y)
\rightarrow (x, 0)$ of $\mathcal {Z}$. Since $P -\beta\pi_{k}Q$ is positive
in $\mathcal {Z}_{I}$, we have
\begin{align*}
\beta\pi_{k}\|x-x'\|^{2}&=\beta\pi_{k}\langle(z-z'),Q(z-z')\rangle_{I}\\
&\leq\langle(z-z'),P(z-z')\rangle_{I}=\|z-z'\|^{2}_{P}.\tag{3.10}
\end{align*}
Putting together (3.9) and (3.10), we get
$$\pi_{k}\|M_{2}^{k}(z)-M_{2}^{k}(z')\|_{P}\leq\|z-z'\|^{2}_{P}.\eqno{(3.11)}$$
So that $\pi_{k}M_{2}^{k}$ is nonexpansive in $\mathcal {Z}_{P}$ .
Let us define on $\mathcal {Z}_{P}$ the function $J : (x, y)
\rightarrow P^{-1}_{k}f(x)$. Then, in $\mathcal {Z}_{P}$ , $\nabla J
= M_{2}^{k}$. Therefore, from Lemma 2.2, $\pi_{k}M_{2}^{k}$ is
firmly nonexpansive in $\mathcal {Z}_{P}$. Then
$T_{2}^{k}\in\mathcal {A}(\mathcal {Z}_{P},\frac{1}{2\pi_{k}})$ [3,
Proposition 4.33]. Hence, feom Lemma 2.3, we know $T^{k}\in\mathcal
{A}(\mathcal {Z}_{P},\frac{1}{\delta_{k}})$, and
$\delta_{k}=2-\frac{1}{2\pi_{k}}$.

Next, we will prove the convergence of Algorithm 2.\\
Since for each $n$, $T^{k}$ is $\frac{1}{\delta_{k}}$-averaged. Therefore, we can write
$$T^{k}=(1-\frac{1}{\delta_{k}})I+\frac{1}{\delta_{k}}S^{k},\eqno{(3.12)}$$
where $S^{k}$ is nonexpansive and $\frac{1}{\delta_{k}}\in]\frac{1}{2},1]$. Then we can rewrite (3.8) as
$$z^{k+1}=(1-\frac{\rho_{k}}{\delta_{k}})z^{k}+\frac{\rho_{k}}{\delta_{k}}S^{k}z^{k}=(1-\alpha_{k})z^{k}+\alpha_{k}S^{k}z^{k},\eqno{(3.13)}$$
where $\alpha_{k}=\frac{\rho_{k}}{\delta_{k}}$. Let $\hat{z} \in Fix(S)$, where $\hat{z}=(\hat{x}, \hat{y})$, then $\hat{x}$ is a solution of (1.1), noticing that $S^{k}\hat{z}=\hat{z}$, we have
\begin{align*}
\|z^{k+1}-\hat{z}\|^{2}_{P}&=(1-\alpha_{k})\|z^{k}-\hat{z}\|^{2}_{P}+\alpha_{k}\|S^{k}z^{k}-\hat{z}\|^{2}_{P}-\alpha_{k}(1-\alpha_{k})\|z^{k}-S^{k}z^{k}\|^{2}_{P}\\
&\leq\|z^{k}-\hat{z}\|^{2}_{P}-\alpha_{k}(1-\alpha_{k})\|z^{k}-S^{k}z^{k}\|^{2}_{P}.\tag{3.14}
\end{align*}
Which implies that
$$\|z^{k+1}-\hat{z}\|^{2}_{P}\leq\|z^{k}-\hat{z}\|^{2}_{P}.\eqno{(3.15)}$$
This implies that sequence $\{z^{k}\}_{k=0}^{\infty}$ is a Fej\'{e}r monotone sequence, and $\lim_{k\rightarrow\infty}\|z^{k+1}-\hat{z}\|_{P}$
exists.

 From the condition (iii) of Theorem 3.1, it is easy to find  that
$$0<\liminf_{k\rightarrow\infty}\alpha_{k}\leq\limsup_{k\rightarrow\infty}\alpha_{k}<1. $$
Therefor, there exists $\underline{a}, \overline{a}\in(0,1)$ such that $\underline{a}<\alpha_{k}< \overline{a}$.
By (3.14), we know
\begin{align*}
\underline{a}(1-\overline{a})\|z^{k}-S^{k}z^{k}\|^{2}_{P}&\leq\alpha_{k}(1-\alpha_{k})\|z^{k}-S^{k}z^{k}\|^{2}_{P}\\
&\leq\|z^{k}-\hat{z}\|^{2}_{P}-\|z^{k+1}-\hat{z}\|^{2}_{P}.
\end{align*}
Hence
$$\lim_{k\rightarrow\infty}\|z^{k}-S^{k}z^{k}\|_{P}=0.\eqno{(3.16)}$$

Since the sequence $\{z^{k}\}$ is bounded and there exists a convergent subsequence $\{z^{k_{j}}\}$
such that
$$z^{k_{j}}\rightarrow\tilde{z},\eqno{(3.17)}$$
for some $\tilde{z}\in\mathcal {X}\times\mathcal {Y}$.\\
From (3.14), we have
$$\lim_{j\rightarrow\infty}\|z^{k_{j}}-S^{k_{j}}z^{k_{j}}\|_{P}=0.\eqno{(3.18)}$$
Since the sequence $\tau_{k}$   is
bounded, there exists a  subsequence  $\tau_{k_{j}}\subset\tau_{k}$ such that
$\frac{1}{\tau_{k_{j}}}\rightarrow\frac{1}{\tau}$. With the same idea, we have
 $\sigma_{k_{j}}\rightarrow\sigma$. Then we obtain that $\delta=2-\frac{1}{2\pi}\in[1,2[$.
Therefor, we know that $T=(I+P^{-1}\circ A)^{-1}\circ(I-P^{-1}\circ B)$ is $\frac{1}{\delta}$-averaged.
So there exists a nonexpansive mapping $S$ such that
$$ T=(I+P^{-1}\circ A)^{-1}\circ(I-P^{-1}\circ B)=(1-\frac{1}{\delta})I+\frac{1}{\delta}S,$$
where $\delta_{k_{j}}\rightarrow\delta$. Because the solution of the
Problem (1.1) is consistent, we know  that
$\bigcap_{k=1}^{\infty}Fix(S^{k})=Fix(S)\neq\emptyset$. Then we will
prove $\lim_{j\rightarrow\infty}\|z^{k_{j}}-Sz^{k_{j}}\|_{P}=0.$ In
fact, we have
\begin{align*}
\|z^{k_{j}}-Sz^{k_{j}}\|_{P}&\leq\|z^{k_{j}}-S^{k_{j}}z^{k_{j}}\|_{P}+\|S^{k_{j}}z^{k_{j}}-Sz^{k_{j}}\|_{P}\\
&=\|z^{k_{j}}-S^{k_{j}}z^{k_{j}}\|_{P}+\|(1-\delta_{k_{j}})z^{k_{j}}+\delta_{k_{j}}T^{k_{j}}z^{k_{j}}-(1-\delta)z^{k_{j}}-\delta Tz^{k_{j}}\|_{P}\\
&\leq\|z^{k_{j}}-S^{k_{j}}z^{k_{j}}\|_{P}+|\delta_{k_{j}}-\delta|(\|z^{k_{j}}\|_{P}+\|Tz^{k_{j}}\|_{P})+\delta\|T^{k_{j}}z^{k_{j}}-Tz^{k_{j}}\|_{P}.\tag{3.19}
\end{align*}
Since $(I+P^{-1}_{k_{j}}\circ A)^{-1}=J_{P^{-1}_{k_{j}}A}$, $(I+P^{-1}\circ A)^{-1}=J_{P^{-1}A}$, so from Lemma 2.5 , we know that
\begin{align*}
\|T^{k_{j}}z^{k_{j}}-Tz^{k_{j}}\|_{P}&=\|(I+P^{-1}_{k_{j}}\circ A)^{-1}\circ(I-P^{-1}_{k_{j}}\circ B)z^{k_{j}}\\
&-(I+P^{-1}\circ A)^{-1}\circ(I-P^{-1}\circ B)z^{k_{j}}\|_{P}\\
&\leq\|J_{P^{-1}_{k_{j}}A}\circ(I-P^{-1}_{k_{j}}\circ B)z^{k_{j}}-J_{P^{-1}_{k_{j}}A}\circ(I-P^{-1}\circ B)z^{k_{j}}\|_{P}\\
&+\|J_{P^{-1}_{k_{j}}A}\circ(I-P^{-1}\circ B)z^{k_{j}}-J_{P^{-1}A}\circ(I-P^{-1}\circ B)z^{k_{j}}\|_{P}\\
&\leq|P^{-1}_{k_{j}}-P^{-1}|\| Bz^{k_{j}}\|_{P}+\| J_{P^{-1}A}(\frac{P^{-1}}{P^{-1}_{k_{j}}}(I-P^{-1}\circ B)z^{k_{j}}\\
&+(1-\frac{P^{-1}}{P^{-1}_{k_{j}}})J_{P^{-1}_{k_{j}}A}\circ(I-P^{-1}\circ B)z^{k_{j}})-J_{P^{-1}A}\circ(I-P^{-1}\circ B)z^{k_{j}}\|_{P}\\
&\leq|P^{-1}_{k_{j}}-P^{-1}|\| Bz^{k_{j}}\|_{P}+\| \frac{P^{-1}}{P^{-1}_{k_{j}}}(I-P^{-1}\circ B)z^{k_{j}}\\
&+(1-\frac{P^{-1}}{P^{-1}_{k_{j}}})J_{P^{-1}_{k_{j}}A}\circ(I-P^{-1}\circ B)z^{k_{j}}-(I-P^{-1}\circ B)z^{k_{j}}\|_{P}\\
&\leq|P^{-1}_{k_{j}}-P^{-1}|\| Bz^{k_{j}}\|_{P}+|1-\frac{P^{-1}}{P^{-1}_{k_{j}}}|\|J_{P^{-1}_{k_{j}}A}\circ(I-P^{-1}\circ B)z^{k_{j}}\\
&-(I-P^{-1}\circ B)z^{k_{j}}\|_{P}.\tag{3.20}
\end{align*}
Put (3.20) into (3.19), we obtain that
\begin{align*}
\|z^{k_{j}}-Sz^{k_{j}}\|_{P}&\leq\|z^{k_{j}}-S^{k_{j}}z^{k_{j}}\|_{P}+|\delta_{k_{j}}-\delta|(\|z^{k_{j}}\|_{P}+\|Tz^{k_{j}}\|_{P})\\
&+\delta|P^{-1}_{k_{j}}-P^{-1}|\| Bz^{k_{j}}\|_{P}+\delta|1-\frac{P^{-1}}{P^{-1}_{k_{j}}}|\|J_{P^{-1}_{k_{j}}A}\circ(I-P^{-1}\circ B)z^{k_{j}}\\
&-(I-P^{-1}\circ B)z^{k_{j}}\|_{P}.\tag{3.21}
\end{align*}
From $\delta_{k_{j}}\rightarrow\delta$ , $P^{-1}_{k_{j}}\rightarrow P^{-1}$ and (3.18),   we have
$$\lim_{j\rightarrow\infty}\|z^{k_{j}}-Sz^{k_{j}}\|_{P}=0.\eqno{(3.22)}$$
By Lemma 2.4, we know $\tilde{z}\in Fix(S)$. Moreover, we know that $\{\|z^{k}-\hat{z}\|_{P}\}$ is non-increasing
for any fixed point $\hat{z}$ of $S$. In particular, by choosing $\hat{z} = \tilde{z}$, we have $\{\|z^{k}-\tilde{z}\|_{P}\}$ is non-
increasing. Combining this and (3.17) yields
$$\lim_{k\rightarrow\infty}z^{k}=\tilde{z}.\eqno{(3.23)}$$
Writing $\tilde{z}=(\tilde{x},\tilde{y})$, then we have $\tilde{x}$
is the solution of Problem (1.1).

\end{proof}
Proof of Theorem 3.3 for Algorithm 3. Before providing the proof of
Theorem 3.3, let us introduce the following notation and Lemma.
\begin{lem}
 Given a Euclidean space $\mathcal {E}$, consider
the minimization problem $\min_{\lambda\in\mathcal{E}}
\bar{f}(\lambda)+\bar{g}(\lambda)+ h(\lambda)$, where $\bar{g}, h
\in \Gamma_{0}(\mathcal {E})$ and where $\bar{f}$ is convex and
differentiable on $\mathcal {E}$ with a L-Lipschitz continuous
gradient. Assume that the infimum is attained and that $0 \in ri(dom
h-dom\bar{g})$. Let $\liminf_{k\rightarrow\infty}\mu_{k}>0$,
$\liminf_{k\rightarrow\infty}\tau_{k}>0$ be such that
$\frac{1}{\liminf_{k\rightarrow\infty}\tau_{k}}-\frac{1}{\liminf_{k\rightarrow\infty}\sigma_{k}}>\frac{L}{2}$.
 , and consider the iterates
$$
\left\{
\begin{array}{l}
y^{k+1}=prox_{\mu_{k}^{-1} h^{\ast}}(y^{k}+\mu_{k}^{-1} \lambda^{k}),~~~~~~~~~~~~~~~~~~~~~~~~~~~~~~~~~~~~~~~~~~~~~~~~~~~~~~~~(3.24a)\\
\lambda^{k+1}=prox_{\tau_{k} \bar{g}}(\lambda^{k}-\tau_{k} \nabla
\bar{f}(\lambda^{k})-\tau_{k}
(2y^{k+1}-y^{k})).~~~~~~~~~~~~~~~~~~~~~~~~~~~~~~~~~~(3.24b)
\end{array}
\right.
$$
Then for any initial value $(\lambda^{0}, y^{0}) \in \mathcal
{E}\times \mathcal {E}$, the sequence $(\lambda^{k}, y^{k})$
converges to a primal-dual point $(\tilde{\lambda}, \tilde{y})$,
i.e., a solution of the equation
$$\min_{\lambda\in\mathcal{E}} \bar{f}(\lambda)+\bar{g}(\lambda)+
h(\lambda)=-\min_{y\in\mathcal{E}} (\bar{f}+\bar{g})^{\ast}(y)+
h^{\ast}(y).\eqno{(3.25)}$$
\end{lem}
\begin{proof}
It is easy to see that the Lemma 3.2 is a special case of Theorem
3.2. So we can obtain Lemma 3.2 from Theorem 3.2 directly.

\end{proof}
Elaborating on Lemma 3.2, we are now ready to establish the Theorem
3.3.

By setting $\mathcal {E} = \mathcal {S}$ and by assuming that
$\mathcal {E}$ is equipped with the same inner product as $\mathcal
{Y}$, one can notice that the functions $\bar{ f} = f \circ D^{-1}$,
$\bar{ g} = g \circ D^{-1}$ and $h$ satisfy the conditions of Lemma
3.2. Moreover, since $(\bar{f}+\bar{g})^{\ast}= (f + g)^{\ast}\circ
D^{\ast}$, one can also notice that $(\tilde{x}, \tilde{y})$ is a
primal-dual point associated with Eq. (2.1) if and only if
$(D\tilde{x}, \tilde{y})$ is a primal-dual point associated with Eq.
(3.25). With the same idea for the proof of Theorem 1 of  [2], we
can recover the ADMMDS$^{+}$ from the iterations (3.24).
\subsection{Connections to other algorithms}
We will further establish the connections to other existing methods.

When $\mu_{k}\equiv\mu$ and $\tau_{k}\equiv\tau$, the ADMMDS$^{+}$
boils down to the ADMM$^{+}$ whose iterations are given by:
$$
\left\{
\begin{array}{l}
z^{k+1}=argmin_{w\in\mathcal {Y}}[h(w)+\frac{\|w-(Dx^{k}+\mu y^{k})\|^{2}}{2\mu}],\\
y^{k+1}=y^{k}+\mu^{-1}(Dx^{k}-z^{k+1}),\\
u^{k+1}=(1-\tau\mu^{-1})Dx^{k}+\tau\mu^{-1}z^{k+1},\\
x^{k+1}=argmin_{w\in\mathcal {X}}[g(w)+\langle\nabla
f(x^{k}),w\rangle+\frac{\|Dw-u^{k+1}-\tau y^{k+1}\|^{2}}{2\tau}].
\end{array}
\right.
$$
In the special case   $h \equiv 0$ , $D = I$, $\mu_{k}\equiv\mu$ and
$\tau_{k}\equiv\tau$ it can be easily verified that $y^{k}$ is null
for all $k \geq 1$ and $u^{k} = x^{k}$. Then, the ADMMDS$^{+}$ boils
down to the standard Forward-Backward algorithm whose iterations are
given by:
\begin{align*}
x^{k+1}&=argmin_{w\in\mathcal
{X}}g(w)+\frac{1}{2\tau}\|w-(x^{k}-\tau \nabla
f(x^{k}))\|^{2}\\
&=prox_{\tau g}(x^{k}-\tau \nabla f(x^{k})).
\end{align*}
One can remark that $\mu$  has disappeared thus it can be set as
large as wanted so the condition on stepsize $\tau$ from Theorem 3.3
boils down to $\tau < 2/L$. Applications of this algorithm with
particular functions appear in well known learning methods such as
ISTA [11].
\section{Coordinate descent}
\subsection{Randomized krasnosel'skii-mann iterations}
 Consider the space
$\mathcal{Z}=\mathcal{Z}_{1}\times\cdots\times\mathcal{Z}_{J}$ for
some $J\in\mathbb{N}^{\ast}$ where for any $j$, $\mathcal{Z}_{j}$ is
a Euclidean space. For $\mathcal{Z}$ equipped with the scalar
product $\langle x,y\rangle=\sum_{j=1}^{J}\langle
x_{j},y_{j}\rangle_{\mathcal{Z}_{j}}$ where $\langle
\cdot,\cdot\rangle_{\mathcal{Z}_{j}}$ is the scalar product in
$\mathcal{Z}_{j}$. For $j\in \{1,\cdots,J\}$ , let $T_{j}:
\mathcal{Z}\rightarrow\mathcal{Z}_{j}$ be the components of the
output of operator $T : \mathcal{Z}\rightarrow\mathcal{Z}$
corresponding to $\mathcal{Z}_{j}$ , so, we have
$Tx=(T_{1}x,\cdots,T_{J}x)$. Let $2^{\mathcal{J}}$ be the power set
of $\mathcal{J}=\{1,\cdots,J\}$. For any $\vartheta\in
2^{\mathcal{J}}$, we donate the operator $\hat{T}^{\vartheta}:
\mathcal{Z}\rightarrow\mathcal{Z}$ by
$\hat{T}^{\vartheta}_{j}x=T_{j}x$ for $j\in\vartheta$ and
$\hat{T}^{\vartheta}_{j}x=x_{j}$ for otherwise. On some probability
space $(\Omega, \mathcal{F}, \mathbb{P})$, we introduce a random
i.i.d. sequence $(\zeta^{k})_{k\in\mathbb{N}^{\ast}}$ such that
$\zeta^{k}:\Omega\rightarrow 2^{\mathcal{J}}$ i.e.
$\zeta^{k}(\omega)$ is a subset of $\mathcal{J}$. Assume that the
following holds:
$$\forall j\in \mathcal{J},\exists \vartheta\in 2^{\mathcal{J}}, j\in\vartheta~~~ and ~~~\mathbb{P}(\zeta_{1}=\vartheta)>0.\eqno{(4.1)}$$

\begin{lem}
(Theorem 3 of [2]). Let $T: \mathcal{Z}\rightarrow\mathcal{Z}$ be
$\eta$-averaged and Fix(T)$\neq\emptyset$. Let
$(\zeta^{k})_{k\in\mathbb{N}^{\ast}}$ be a random i.i.d. sequence on
$2^{\mathcal{J}}$ such that Condition (4.1) holds. If for all $k$,
sequence $(\beta_{k})_{k\in\mathbb{N}}$ satisfies
$$0<\liminf_{k\rightarrow\infty}\beta_{k}\leq\limsup_{k\rightarrow\infty}\beta_{k}<\frac{1}{\eta}.$$
Then, almost surely, the iterated sequence

$$x^{k+1}=x^{k}+\beta_{k}(\hat{T}^{(\zeta^{k+1})}x^{k}-x^{k})\eqno{(4.2)}$$
converges to some point in Fix($T$).
\end{lem}
\subsection{Randomized Modified krasnosel'skii- mann iterations}

\begin{thm}
 Let $T$  be
$\eta$-averaged and $T^{k}$ be $\eta_{k}$-averaged on $\mathcal{Z}$
and $\bigcap_{k=1}^{\infty}$Fix(T$^{k}$)=Fix(T)$\neq\emptyset$,  and
$T^{k}\rightarrow T$. Let $(\zeta^{k})_{k\in\mathbb{N}^{\ast}}$ be a
random i.i.d. sequence on $2^{\mathcal{J}}$ such that Condition
(4.1) holds. If for all $k$, sequence $(\beta_{k})_{k\in\mathbb{N}}$
satisfies
$$0<\beta_{k}<\frac{1}{\eta_{k}},0<\liminf_{k\rightarrow\infty}\beta_{k}\leq\limsup_{k\rightarrow\infty}\beta_{k}<\frac{1}{\limsup_{k\rightarrow\infty}\eta_{k}}.$$
Then, almost surely, the iterated sequence

$$x^{k+1}=x^{k}+\beta_{k}(\hat{T}^{k,(\zeta^{k+1})}x^{k}-x^{k})\eqno{(4.3)}$$
converges to some point in Fix($T$).
\end{thm}

\begin{proof}
Define the operator $U^{k}=(1-\beta_{k})I+\beta_{k}T^{k}$;
similarly, define
$U^{k,(\vartheta)}=(1-\beta_{k})I+\beta_{k}T^{k,(\vartheta)}$.
Observing that the operator $U^{k}$ is $(\beta_{k}\eta_{k})$-
averaged. From (4.3), we  can know that
$x^{k+1}=U^{k,(\zeta^{k+1})}x^{k}$. Set
$p_{\vartheta}=\mathbb{P}(\zeta_{1}=\vartheta)$ for any
$\vartheta\in 2^{\mathcal{J}}$. Denote by $\|x\|^{2} = \langle x,
x\rangle$ the squared norm in $\mathcal{Z}$. Define a new inner
product $x\bullet y=\sum_{j=1}^{J}q_{j}\langle x_{j},
y_{j}\rangle_{j}$ on $\mathcal{Z}$ where
$q_{j}^{-1}=\sum_{\vartheta\in2^{\mathcal{J}}}p_{\vartheta}\mathbf{1}_{\{j\in\vartheta\}}$
and let $\||x\||^{2}= x\bullet x$ be its associated squared norm.
Consider any $x^{\ast}\in Fix(T)$. Conditionally to the sigma-field
$\mathcal {F}^{k} = \sigma(\zeta_{1}, . . . , \zeta^{k})$ we have
\begin{align*}
\mathbb{E}[\||x^{k+1}-\tilde{x}\||^{2}|\mathcal
{F}^{k}]&=\sum_{\vartheta\in2^{\mathcal{J}}}p_{\vartheta}\||U^{k,(\vartheta)}x^{k}-\tilde{x}\||^{2}\\
&=\sum_{\vartheta\in2^{\mathcal{J}}}p_{\vartheta}\sum_{j\in\vartheta}q_{j}\|U^{k}_{j}x^{k}-\tilde{x}_{j}\|^{2}+\sum_{\vartheta\in2^{\mathcal{J}}}p_{\vartheta}\sum_{j\neq\vartheta}q_{j}\|x^{k}_{j}-\tilde{x}_{j}\|^{2}\\
&=\||x^{k}-\tilde{x}\||^{2}+\sum_{\vartheta\in2^{\mathcal{J}}}p_{\vartheta}\sum_{j\in\vartheta}q_{j}(\|U^{k}_{j}x^{k}-\tilde{x}_{j}\|^{2}-\|x^{k}_{j}-\tilde{x}_{j}\|^{2})\\
&=\||x^{k}-\tilde{x}\||^{2}+\sum_{j=1}^{J}(\|U^{k}_{j}x^{k}-\tilde{x}_{j}\|^{2}-\|x^{k}_{j}-\tilde{x}_{j}\|^{2})\\
&=\||x^{k}-\tilde{x}\||^{2}+(\|U^{k}x^{k}-\tilde{x}\|^{2}-\|x^{k}-\tilde{x}\|^{2})
\end{align*}
Since $U^{k}$ is $(\beta_{k}\eta_{k})$-averaged and that $\tilde{x}$
is a fixed point of $U^{k}$, the term enclosed in the parentheses is
no larger than
$-\frac{1-\beta_{k}\eta_{k}}{\beta_{k}\eta_{k}}\|(I-U^{k})x^{k}\|^{2}$.
By $I-U^{k}=\beta_{k}(I-T^{k})$, we have:

\begin{align*}
\mathbb{E}[\||x^{k+1}-\tilde{x}\||^{2}|\mathcal {F}^{k}]
&\leq\||x^{k}-\tilde{x}\||^{2}-\beta_{k}(1-\beta_{k}\eta_{k})\|(I-T^{k})x^{k}\|^{2}\\
&\leq\||x^{k}-\tilde{x}\||^{2}-\beta_{k}\eta_{k}(1-\beta_{k}\eta_{k})\|(I-T^{k})x^{k}\|^{2},\tag{4.4}
\end{align*}
which shows that $\||x^{k}-\tilde{x}\||^{2}$ is a nonnegative
supermartingale with respect to the filtration $(\mathcal {F}^{k})$.
As such, it converges with probability one towards a random variable
that is finite almost everywhere.

Given a countable dense subset $\mathrm{Z}$ of $Fix(T)$, there is a
probability one set on which $\||x^{k}-x\||\rightarrow
\mathrm{X}_{x}\in[0,\infty)$ for all $x\in\mathrm{Z}$. Let $x\in
Fix(T)$, let $\varepsilon > 0$, and choose $x\in\mathrm{Z}$ such
that $\||\tilde{x}-x\||\leq\varepsilon$. With probability one, we
have
$$\||x^{k}-\tilde{x}\||\leq\||x^{k}-x\||+\||\tilde{x}-x\||\leq\mathrm{X}_{x}+2\varepsilon,$$
for $k$ large enough. Similarly
$\||x^{k}-\tilde{x}\||\geq\mathrm{X}_{x}-2\varepsilon$, for $k$
large enough. Therefor, we have

$\mathbf{A}_{1}$: There is a probability one set on which
$\||x^{k}-\tilde{x}\||$ converges for every $\tilde{x}\in Fix(T)$.\\
From the assumption on $(\beta_{k})_{k\in\mathbb{N}}$, we know that
$0<\liminf_{k\rightarrow\infty}\beta_{k}\eta_{k}\leq\limsup_{k\rightarrow\infty}\beta_{k}\eta_{k}<1.$
So there exits$\overline{a}, \underline{a}\in (0,1)$, such that
$\underline{a}<\beta_{k}\eta_{k}<\overline{a}$. From (4.4), we have
\begin{align*}
\underline{a}(1-\overline{a})\|(I-T^{k})x^{k}\|^{2}&\leq\alpha_{k}\beta_{k}(1-\alpha_{k}\beta_{k})\|(I-T^{k})x^{k}\|^{2}\\
&\leq\||x^{k}-\tilde{x}\||^{2}-\mathbb{E}[\||x^{k+1}-\tilde{x}\||^{2}|\mathcal
{F}^{k}].\tag{4.5}
\end{align*}

 Taking the expectations on both sides of  inequality (4.5) and iterating over $k$, we
obtain
$$\mathbb{E}\|(I-T^{k})x^{k}\|^{2}\leq\frac{1}{\underline{a}(1-\overline{a})}(x^{0}-\tilde{x})^{2}.$$
By Markov¡¯s inequality and Borel Cantelli¡¯s lemma,we therefore
obtain:

$\mathbf{A}_{2}$: $(I-T^{k})x^{k}\rightarrow0$ almost surely.\\
 We now consider an elementary event in the probability one set where $\mathbf{A}_{1}$
and $\mathbf{A}_{2}$ hold. On this event, since the sequence
$(x^{k})_{k\in \mathbb{N}}$ is bounded, so there exists a convergent
subsequence $(x^{k_{j}})_{j\in \mathbb{N}}$ such that
\begin{align*}
\lim_{j\rightarrow\infty}\||x^{k_{j}}-\hat{x}\||=0,\tag{4.6}
\end{align*}
for some $\hat{x}\in\mathcal {Z}$.\\
From $\mathbf{A}_{2}$ and the condition $T^{k}\rightarrow T$, we
have
\begin{align*}
\|x^{k_{j}}-Tx^{k_{j}}\|&\leq\|x^{k_{j}}-T^{k_{j}}x^{k_{j}}\|+\|T^{k_{j}}x^{k_{j}}-Tx^{k_{j}}\|\\
&\leq\|x^{k_{j}}-T^{k_{j}}x^{k_{j}}\|+|T^{k_{j}}-T|\|x^{k_{j}}\|\rightarrow0.\tag{4.7}
\end{align*}
It then follows from Lemma 2.4 that $\hat{x}\in Fix(T)$. Moreover,
we know that on this event,  $\||x^{k}-\tilde{x}\||$ converges for
any $\tilde{x}\in Fix(T)$. In particular, by choosing
$\tilde{x}=\hat{x}$, we see that  $\||x^{k}-\hat{x}\||$  converges.
Combining this and (4.6) yields
$$\lim_{k\rightarrow\infty}\||x^{k}-\hat{x}\||=0.$$

\end{proof}

From Theorem 3.3, we know that the ADMMDS$^{+}$ iterates are
generated by the action of a $\eta_{k}$-averaged operator. Theorem
4.1 shows then that a stochastic coordinate descent version of the
ADMMDS$^{+}$ converges towards a primal-dual point. This result will
be exploited in two directions: first, we describe a stochastic
minibatch algorithm, where a large dataset is randomly split into
smaller chunks. Second, we develop an asynchronous version of the
ADMMDS$^{+}$ in the context where it is distributed on a graph.

\section{Application to stochastic approximation}
\subsection{Problem setting}
Given an integer $N > 1$, consider the problem of minimizing a sum
of composite functions

$$\inf_{x\in\mathcal{X}}\sum_{n=1}^{N}(f_{n}(x)+g_{n}(x)),\eqno{(5.1)}$$

where we make the following assumption:

\begin{ass}
For each $n = 1, ...,N$,\\
 (1) $f_{n}$ is a convex differentiable function on $\mathcal{X}$, and its
gradient $\nabla f_{n}$ is  $1/\beta$-Lipschitz continuous on
$\mathcal{X}$ for some $\beta\in(0,+\infty)$;\\
 (2) $g_{n}\in \Gamma_{0}(\mathcal{X})$;\\
 (3) The infimum of Problem (5.1) is attained;\\
 (4) $\cap_{n=1}^{N}ridom g_{n}\neq0.$
\end{ass}

This problem arises for instance in large-scale learning
applications where the learning set is too large to be handled as a
single block. Stochastic minibatch approaches consist in splitting
the data set into $N$ chunks and to process each chunk in some
order, one at a time. The quantity $f_{n}(x) + g_{n}(x)$ measures
the inadequacy between the model (represented by parameter $x$) and
the $n$-th chunk of data. Typically, $f_{n}$ stands for a data
fitting term whereas $g_{n}$ is a regularization term which
penalizes the occurrence of erratic solutions. As an example, the
case where $f_{n}$ is quadratic and $g_{n}$ is the $l_{1}$-norm
reduces to the popular LASSO problem [12]. In particular,  it also
useful to recover sparse signal.

\subsection{ Instantiating the ADMMDS$^{+}$}
We regard our stochastic minibatch algorithm as an instance of the
ADMMDS$^{+}$ coupled with a randomized coordinate descent. In order
to end that ,we rephrase Problem (5.1) as
$$\inf_{x\in\mathcal{X}^{N}}\sum_{n=1}^{N}(f_{n}(x)+g_{n}(x))+\iota_{\mathcal{C}}(x),\eqno{(5.2)}$$
where the notation $x_{n}$ represents the $n$-th component of any $x
\in \mathcal{X}^{N}$, $\mathcal{C}$ is the space of vectors $x \in
\mathcal{X}^{N}$ such that $x_{1} = \cdots = x_{N}$. On the space
$\mathcal{X}^{N}$, we set $f(x) = \sum_{ n} f_{n}(x_{n})$, $g(x) =
\sum_{ n} g_{n}(x_{n})$, $h(x) = \iota_{\mathcal{C}}$ and $D =
I_{\mathcal{X}^{N}}$ the identity matrix. Problem (5.2) is
equivalent to
$$\min_{x\in\mathcal{X}^{N}} f(x)+g(x)+ (h\circ D)(x).\eqno{(5.3)}$$
We define the natural scalar product on $\mathcal{X}^{N}$ as
$\langle x,y\rangle=\sum_{n=1}^{N}\langle x_{n},y_{n}\rangle$.
Applying the ADMMDS$^{+}$ to solve Problem (5.3) leads to the
following iterative scheme:
\begin{align*}
&z^{k+1}=proj_{\mathcal{C}}\|x^{k}+\mu_{k}y^{k}\|^{2}, \\
&y^{k+1}_{n}=y^{k}_{n}+\mu_{k}^{-1}(x^{k}_{n}-z^{k+1}_{n} ),\\
&u^{k+1}_{n}=(1-\tau_{k}\mu_{k}^{-1})x^{k}_{n}+\tau_{k}\mu_{k}^{-1}z^{k+1}_{n},\\
&x^{k+1}_{n}=\arg\min_{w\in\mathcal {X}}[g_{n}(w)+\langle\nabla
f_{n}(x^{k}),w\rangle+\frac{\|w-u^{k+1}_{n}-\tau_{k}y^{k+1}_{n}\|^{2}}{2\tau_{k}}],
\end{align*}
where $proj_{\mathcal{C}}$ is the orthogonal projection onto
$\mathcal{C}$. Observe that for any $x \in \mathcal{X}^{N}$,
$proj_{\mathcal{C}}(x)$ is equivalent to $(\bar{x}, \cdots ,
\bar{x})$ where $\bar{x}$ is the average of vector $x$, that is
$\bar{x}=N^{-1}\sum_{n}x_{n}$. Consequently, the components of
$z^{k+1}$ are equal and coincide with $\bar{x}^{k}+\mu_{k}
\bar{y}^{k}$ where $\bar{x}^{k}$ and $\bar{y}^{k}$ are the averages
of $x^{k}$  and $y^{k}$ respectively. By inspecting the $y^{k}$
$n$-update equation above, we notice that the latter equality
simplifies even further by noting that $\bar{y}^{k+1} = 0$ or,
equivalently, $\bar{y}^{k} = 0$ for all $k\geq 1$ if the algorithm
is started with $\bar{y}^{0} = 0$. Finally, for any $n$ and $k\geq
1$, the above iterations reduce to
\begin{align*}
&\bar{x}^{k}=\frac{1}{N}\sum_{n=1}^{N} x^{k}_{n},\\
&y^{k+1}_{n}=y^{k}_{n}+\mu_{k}^{-1}(x^{k}_{n}-\bar{x}^{k} ),\\
&u^{k+1}_{n}=(1-\tau_{k}\mu_{k}^{-1})x^{k}_{n}+\tau_{k}\mu_{k}^{-1}\bar{x}^{k},\\
&x^{k+1}_{n}=prox_{\tau_{k}g_{n}}[u^{k+1}_{n}-\tau_{k}(\nabla
f_{n}(x^{k}_{n})+y^{k+1}_{n})].
\end{align*}
These iterations can be written more compactly as
\begin{algorithm}[H]
\caption{Minibatch ADMMDS$^{+}$.}
\begin{algorithmic}\label{algorithm4}
\STATE Initialization: Choose $x^{0}\in \mathcal{X}$, $y^{0}\in
\mathcal{Y}$, s.t. $\sum_{n}v^{0}_{n}=0$.\\
Do
$$
\begin{array}{l}
\bullet ~~\bar{x}^{k}=\frac{1}{N}\sum_{n=1}^{N} x^{k}_{n},\\
\bullet ~~For~ batches~  n = 1, \cdots ,N,~  do\\
~~~~y^{k+1}_{n}=y^{k}_{n}+\mu_{k}^{-1}(x^{k}_{n}-\bar{x}^{k} ),\\
~~~~x^{k+1}_{n}=prox_{\tau_{k}g_{n}}[(1-2\tau_{k}\mu_{k}^{-1})x^{k}_{n}-\tau_{k}\nabla
f_{n}(x^{k}_{n})+2\tau_{k}\mu_{k}^{-1}\bar{x}^{k}-\tau_{k}y^{k}_{n}].\\
\bullet ~~Increment~ k.
\end{array}\eqno{(5.4)}
$$

\end{algorithmic}
\end{algorithm}
The following result is a straightforward consequence of Theorem
3.3.
\begin{thm}
Assume that the minimization Problem (5.3) is consistent,
$\liminf_{k\rightarrow\infty}\mu_{k}>0$, and
$\liminf_{k\rightarrow\infty}\tau_{k}>0$.  Let Assumption 5.1  hold
true and
$\frac{1}{\liminf_{k\rightarrow\infty}\tau_{k}}-\frac{1}{\liminf_{k\rightarrow\infty}\mu_{k}}>\frac{L}{2}$.
Let the sequences $(\bar{x}^{k},y^{k})$ be generated by Minibatch
ADMMDS$^{+}$. Then for any initial point $( x^{0}, y^{0})$ such that
$\bar{y}^{0} = 0$, the sequence $\{\bar{x}^{k}\}$ converges to a
solution of Problem (5.3).
\end{thm}

At each step $k$, the iterations given above involve the whole set
of functions $f_{n}, g_{n} (n = 1, \cdots,N)$. Our aim is now to
propose an algorithm which involves a single couple of functions
$(f_{n}, g_{n})$ per iteration.
\subsection{ A stochastic minibatch primal-dual splitting  algorithm with dynamic
stepsize} We are now in position to state the main algorithm of this
section. The proposed stochastic minibatch primal-dual splitting
algorithm with dynamic stepsize(SMPDSDS) is obtained upon applying
the randomized coordinate descent on the minibatch ADMMDS$^{+}$:

\begin{algorithm}[H]
\caption{SMPDSDS.}
\begin{algorithmic}\label{algorithm5}
\STATE Initialization: Choose $x^{0}\in \mathcal{X}$, $y^{0}\in
\mathcal{Y}$.\\
Do
$$
\begin{array}{l}
\bullet ~~Define ~ \bar{x}^{k}=\frac{1}{N}\sum_{n=1}^{N} x^{k}_{n}, ~ \bar{y}^{k}=\frac{1}{N}\sum_{n=1}^{N}v^{k}_{n},\\
\bullet ~~Pick ~up ~the ~value ~of  ~\zeta^{k+1}, \\
\bullet ~~For~ batch~n=\zeta^{k+1},~set\\
~~~~y^{k+1}_{n}=y^{k}_{n}-\bar{y}^{k}+\frac{(x^{k}_{n}-\bar{x}^{k} )}{\mu_{k}},~~~~~~~~~~~~~~~~~~~~~~~~~~~~~~~~~~~~~~~~~~~~~~~~~~~~~~~~~~~~(5.5a)\\
~~~~x^{k+1}_{n}=prox_{\tau_{k}g_{n}}[(1-2\tau_{k}\mu_{k}^{-1})x^{k}_{n}-\tau_{k}\nabla
f_{n}(x^{k}_{n})-\tau_{k}y^{k}_{n}+2\tau_{k}(\mu_{k}^{-1}\bar{x}^{k}+\bar{y}^{k})].(5.5b)\\
\bullet ~~For~ all ~batches ~n\neq\zeta^{k+1},~~ y^{k+1}_{n}=y^{k}_{n}, x^{k+1}_{n}=x^{k}_{n}.\\
 \bullet ~~Increment~ k.
\end{array}
$$

\end{algorithmic}
\end{algorithm}

\begin{ass}
The random sequence $(\zeta^{k})_{k\in\mathbb{N}^{\ast}}$ is i.i.d.
and satisfies $\mathbb{P}[\zeta^{1} = n] > 0$ for all $n = 1,
...,N$.
\end{ass}

\begin{thm}
Assume that the minimization Problem (5.3) is consistent,
$\liminf_{k\rightarrow\infty}\mu_{k}>0$, and
$\liminf_{k\rightarrow\infty}\tau_{k}>0$.  Let Assumption 5.1 and
Assumption 5.2 hold true and
$\frac{1}{\liminf_{k\rightarrow\infty}\tau_{k}}-\frac{1}{\liminf_{k\rightarrow\infty}\mu_{k}}>\frac{L}{2}$.
 Then for any initial point $( x^{0}, y^{0})$ , the sequence $\{\bar{x}^{k}\}$  generated by SMPDSDS algorithm converges to a
solution of Problem (5.3).
\end{thm}
\begin{proof}
Let us define $(\bar{f}, \bar{g}, h, D)=(f, g, h, I_{x^{N}})$ where
the functions $f$, $g$, and $h$ are the ones defined in section 5.2.
Then the iterates $((y^{k+1}_{ n} )^{N} _{n=1}, (x^{k+1}_{ n} )^{N}
_{n=1})$ described by Equations (5.4) coincide with the iterates $(
y^{k+1}, x^{k+1})$ described by Equations (3.24). If we write these
equations more compactly as $(y^{k+1}, x^{k+1}) = T^{k}(y^{k},
x^{k})$ where  the operator $T^{k}$ acts in the space
$\mathcal{Z}=\mathcal{X}^{N}\times\mathcal{X}^{N}$, then from the
proof of Theorem 3.2, we konw that $T^{k}$ is $\eta_{k}$-averaged,
where $\eta_{k}=(2-\bar{\eta}_{k})^{-1}$ and
$\bar{\eta}_{k}=\frac{L}{2}(\tau_{k}^{-1}-\mu_{k}^{-1})$ .
 Defining the selection operator $\mathcal{S}_{n }$ on
$\mathcal{Z}$ as $\mathcal{S}_{n }(y, x) = (y_{n}, x_{n})$, we
obtain that $\mathcal{Z} = \mathcal{S}_{1 }(\mathcal{Z})\times
\cdots\times \mathcal{S}_{N }(\mathcal{Z})$ up to an element
reordering. To be compatible with the notations of Section 4.1, we
assume that $J = N$ and that the random sequence $\zeta^{k}$ driving
the SMPDSDS algorithm is set valued in
$\{\{1\},\ldots\{N\}\}\subset2^{\mathcal{J}}$. In order to establish
Theorem 5.2, we need to show that the iterates $(y^{k+1}, x^{k+1})$
provided by the SMPDSDS algorithm are those who satisfy the equation
$(y^{k+1}, x^{k+1}) = T^{k,(\zeta^{k+1})}(y^{k}, x^{k})$.
 By the direct application of Theorem 4.1, we can obtain Theorem
 5.2.\\
Let us start with the $y$-update equation. Since $h = \iota_{C}$,
its Legendre-Fenchel transform is $h^{\ast} = \iota_{C^{\perp}}$
where $C^{\perp}$ is the orthogonal complement of $C$ in $\mathcal
{X}^{N}$. Consequently,

If we write $(\varsigma^{k+1}, \upsilon^{k+1}) = T^{k}(y^{k},
x^{k})$, then by Eq. (3.24a),

$$\varsigma^{k+1}_{n}=y^{k}_{n}-\bar{y}^{k}+\frac{(x^{k}_{n}-\bar{x}^{k} )}{\mu_{k}}~~n=1,\ldots N.$$
Observe that in general, $\bar{y}^{k}\neq 0$ because in the SMPDSDS
algorithm, only one component is updated at a time. If $\{n\} =
\zeta^{k+1}$, then $y^{k+1}_{n} = \varsigma^{k+1}_{n}$ which is Eq.
(5.5a). All other components of $y^{k}$ are carried over to $
y^{k+1}$ .\\
 By Equation (3.24b) we also get
$$\upsilon^{k+1}_{n}=prox_{\tau_{k}g_{n}}[x^{k}_{n}-\tau_{k}\nabla
f_{n}(x^{k}_{n})-\tau_{k}(2y^{k+1}_{n}-y^{k})].$$ If $\{n\} =
\zeta^{k+1}$, then $x^{k+1}_{n}=\upsilon^{k+1}_{n}$ can easily be
shown to be given by (5.5b).

\end{proof}

\section{Distributed optimization}

~~~~Consider a set of $N > 1$ computing agents that cooperate to
solve the minimization Problem (5.1). Here, $f_{n}$, $g_{n}$ are two
private functions available at Agent $n$. Our purpose is to
introduce a random distributed algorithm to to solve (5.1). The
algorithm is asynchronous in the sense that some components of the
network are allowed to wake up at random and perform local updates,
while the rest of the network stands still. No coordinator or global
clock is needed. The frequency of activation of the various network
components is likely to vary.

The examples of this problem appear in learning applications where
massive training data sets are distributed over a network and
processed by distinct machines [13], [14], in resource allocation
problems for communication networks [15], or in statistical
estimation problems by sensor networks [16], [17].
\subsection{Network model and problem formulation}
We consider the network as a graph $G = (Q,E)$ where $Q = \{1,
\cdots ,N\}$ is the set of agents/nodes and $E\subset \{1, \cdots
,N\}^{2}$ is the set of undirected edges. We write $n\sim m$
whenever ${n,m}\in E$. Practically, $n\sim m $ means that agents $n$
and $m$ can communicate with each other.

\begin{ass}
$G$ is connected and has no self loop.
\end{ass}

Now we introduce some notations. For any $x\in\mathcal{X}^{|Q|}$, we
denote by $x_{n}$ the components of $x$, i.e., $x = (x_{n})_{n\in
Q}$. We redard the functions $f$ and $g$ on
$\mathcal{X}^{|Q|}\rightarrow(-\infty,+\infty]$ as $f(x)=\sum_{n\in
Q}f_{n}(x_{n})$ and $g(x)=\sum_{n\in Q}g_{n}(x_{n})$. So the Problem
(5.1) is equal to the minimization of $f(x)+g(x)$ under the
constraint that all components of $x$ are equal.

Next we write the latter constraint in a way that involves the graph
$G$. We replace the global consensus constraint by a modified
version of the function $\iota_{\mathcal{C}}$ . The purpose of us is
to ensure global consensus through local consensus over every edge
of the graph.

For any $\epsilon\in E$, say $\epsilon= \{n,m\}\in Q$ , we define
the linear operator $D_{\epsilon}(x) : \mathcal{X}^{|Q|} \rightarrow
\mathcal{X}^{2}$ as $D_{\epsilon}(x) = (x_{n}, x_{m})$ where we
assume some ordering on the nodes to avoid any ambiguity on the
definition of $D$. We construct the linear operator
$D:\mathcal{X}^{|Q|} \rightarrow\mathcal{Y}\triangleq
\mathcal{X}^{2|Q|}$ as $D(x)=(D_{\epsilon}(x))_{\epsilon\in E}$
where we also assume some ordering on the edges. Any vector $y \in
\mathcal{Y}$ will be written as $y = (y_{\epsilon})_{\epsilon\in E}$
where, writing $\epsilon= \{n,m\} \in E$, the component
$y_{\epsilon}$ will be represented by the couple $y_{\epsilon}=
(y_{\epsilon}(n), y_{\epsilon}(m))$ with $n < m$. We also introduce
the subspace of
 $\mathcal{X}^{2}$ defined as $\mathcal{C}_{2} = \{(x, x) : x \in \mathcal{X}\}$. Finally, we define $h : \mathcal{Y} \rightarrow (
-\infty,+\infty]$ as

$$h(y)=\sum_{\epsilon\in E}\iota_{\mathcal{C}_{2}}(y_{\epsilon}).\eqno{(6.1)}$$
Then we consider the following problem:
$$\min_{x\in\mathcal{X}^{|Q|}} f(x)+g(x)+ (h\circ D)(x).\eqno{(6.2)}$$

\begin{lem}
([2]). Let Assumptions 6.1 hold true.  The minimizers of (6.2) are
the tuples $(x^{\ast}, \cdots , x^{\ast})$ where $x^{\ast}$ is any
minimizer of (5.1).
\end{lem}

\subsection{Instantiating the ADMMDS$^{+}$}

Now we use the ADMMDS$^{+}$ to solve the Problem (6.2). Since the
newly defined function $h$ is separable with respect to the
$(y_{\epsilon})_{\epsilon\in E}$, we get

$$prox_{\tau_{k} h}(y)=(prox_{\tau_{k} \iota_{\mathcal{C}_{2}}}(y_{\epsilon}))_{\epsilon\in E}=((\bar{y}_{\epsilon}, \bar{y}_{\epsilon}))_{\epsilon\in E}$$

where $\bar{y}_{\epsilon}=(y_{\epsilon}(n)+y_{\epsilon}(m))/2$ if
$\epsilon=\{n,m\}$. With this at hand, the update equation (a) of
the ADMMDS$^{+}$ can be  written as
$$z^{k+1}=((\bar{z}_{\epsilon}^{k+1},\bar{z}_{\epsilon}^{k+1}))_{\epsilon\in E},$$
where
$$\bar{z}^{k+1}=\frac{x^{k}_{n}+x^{k}_{m}}{2}+\frac{\mu_{k}(y_{\epsilon}^{k}(n)+y_{\epsilon}^{k}(m))}{2}$$
 for any $\epsilon= \{n,m\}\in E$.
  Plugging this equality into Eq. (b) of the
ADMMDS$^{+}$, it can be seen that
$y_{\epsilon}^{k}(n)=-y_{\epsilon}^{k}(m)$. Therefore
$$\bar{z}^{k+1}=\frac{x^{k}_{n}+x^{k}_{m}}{2},$$
for any $k \geq 1$. Moreover
$$y_{\epsilon}^{k+1}=\frac{x^{k}_{n}-x^{k}_{m}}{2\mu_{k}}+y_{\epsilon}^{k}(n).$$
Observe that the $n$-th component of the vector $D^{T}Dx$ coincides
with $ d_{n}x_{n}$, where $ d_{n}$ is the degree (i.e., the number
of neighbors) of node $n$. From (d) of the ADMMDS$^{+}$, the
$n^{th}$ component of  $  x^{k+1}$ can be written
$$x^{k+1}_{n}=prox_{\tau_{k}g_{n}/d_{n}}[\frac{(D^{\ast}(u^{k+1}-\tau_{k}y^{k+1}))_{n}-\tau_{k}\nabla
f_{n}(x^{k}_{n})}{d_{n}}],$$ where for any $y\in\mathcal{Y}$,
$$(D^{T}y)_{n}=\sum_{m:\{n,m\}\in E}y_{\{n,m\}}(n)$$
is the $n$-th component of $D^{T}y \in \mathcal{X}^{|Q |}$. Plugging
Eq.  (c) of the ADMMDS$^{+}$ together with the expressions of
$\bar{z}^{k+1}_{\{n,m\}}$ and $y^{k+1}_{\{n,m\}}$ in the argument of
$prox_{\tau_{k}g_{n}/d_{n}}$ , we can have
\begin{align*}
x^{k+1}_{n}&=prox_{\tau_{k}g_{n}/d_{n}}[(1-\tau_{k}\mu_{k}^{-1})x^{k}_{n}-\frac{\tau_{k}}{d_{n}}\nabla
f_{n}(x^{k}_{n})+\frac{\tau_{k}}{d_{n}}\sum_{m:\{n,m\}\in
E}(\mu_{k}^{-1}x^{k}_{m}-y^{k}_{\{n,m\}}(n))].
\end{align*}
The algorithm is finally described by the following procedure: Prior
to the clock tick $k + 1$, the node $n$ has in its memory the
variables $x^{k}_{n}$, $\{y_{\{n,m\}}^{k}(n)\}_{m\thicksim n}$,
 and  $\{x^{k}_{m}\}_{m\thicksim n}$.
\begin{algorithm}[H]
\caption{Distributed ADMMDS$^{+}$.}
\begin{algorithmic}\label{1}
\STATE Initialization: Choose $x^{0}\in \mathcal{X}$, $y^{0}\in
\mathcal{Y}$, s.t. $\sum_{n}y^{0}_{n}=0$.\\
Do
$$
\begin{array}{l}
\bullet ~~For~ any~ n \in Q ,~ Agent ~n ~performs ~the~ following~
operations: \\
~~~~y_{\{n,m\}}^{k+1}(n)=y_{\{n,m\}}^{k}(n)+\frac{x^{k}_{n}-x^{k}_{m}}{2},~~for~all~m\thicksim n, ~~~~~~~~~~~~~~~~~~~~~~~~~~~~~~~~~~(6.3a)\\
~~~~~~~~~~x^{k+1}_{n}=prox_{\tau_{k}g_{n}/d_{n}}[(1-\tau_{k}\mu_{k}^{-1})x^{k}_{n}-\frac{\tau_{k}}{d_{n}}\nabla
f_{n}(x^{k}_{n})\\
~~~~~~~~~~~~~~~~~~~+\frac{\tau_{k}}{d_{n}}\sum_{m:\{n,m\}\in
E}(\mu_{k}^{-1}x^{k}_{m}-y^{k}_{\{n,m\}}(n))].~~~~~~~~~~~~~~~~~~~~~~~~~~~~~~~~(6.3b)\\
\bullet ~~Agent~ n ~sends~ the~ parameter~ y^{k+1}_{n}, x^{k+1}_{n}
 ~to~ their~ neighbors~ respectively.\\
 \bullet ~~Increment~ k.
\end{array}
$$

\end{algorithmic}
\end{algorithm}

\begin{thm}
 Assume that the minimization Problem (5.1) is consistent,
$\liminf_{k\rightarrow\infty}\mu_{k}>0$, and
$\liminf_{k\rightarrow\infty}\tau_{k}>0$.  Let Assumption 5.1 and
Assumption 6.1 hold true and
$\frac{1}{\liminf_{k\rightarrow\infty}\tau_{k}}-\frac{1}{\liminf_{k\rightarrow\infty}\mu_{k}}>\frac{L}{2}$.
Let $(x^{k})_{k\in\mathbb{N}} $ be the sequence generated by
Distributed ADMMDS$^{+}$ for any initial point $( x^{0}, y^{0})$.
Then for all $n \in Q$ the sequence $(x^{k}_{n})_{k\in \mathbb{N}}$
converges to a solution of Problem (5.1).
\end{thm}

\subsection{A Distributed asynchronous primal-dual splitting  algorithm with dynamic
stepsize} In this section, we use the randomized coordinate descent
on the above algorithm, we call this algorithm as distributed
asynchronous primal-dual splitting  algorithm with dynamic stepsize
(DASPDSDS). This algorithm has the following attractive property:\\
Firstly, at each iteration, a single agent, or possibly a subset of
agents chosen at random, are activated. Moreover, in the algorithm
the coefficient $\tau$, $\sigma$ is made iteration-dependent to
solve the general Problem (5.1), errors are allowed in the
evaluation of the operators $prox_{\sigma h^{\ast}}$, $prox_{\tau
g_{n}}$ and $\nabla f_{n}$. The errors allow for some tolerance in
the numerical implementation of the algorithm, while the flexibility
introduced by the iteration-dependent parameters $\tau_{k}$ and
$\sigma_{k}$  can be used to improve its convergence pattern.
Finally, if we let $(\zeta^{k})_{k\in \mathbb{N}}$ be a sequence of
i.i.d. random variables valued in $2^{ Q}$. The value taken by
$\zeta^{k}$ represents the agents that will be activated and perform
a prox on their $x$ variable at moment $k$. The asynchronous
algorithm goes as follows:
\begin{algorithm}[H]
\caption{DASPDSDS.}
\begin{algorithmic}\label{1}
\STATE Initialization: Choose $x^{0}\in \mathcal{X}$, $y^{0}\in
\mathcal{Y}$.\\
Do
$$
\begin{array}{l}
\bullet ~~Select~ a~ random~ set ~of~ agents~ \zeta^{k+1} =\mathcal{B}. \\
\bullet ~~For~ any~ n \in \mathcal{B},~ Agent~ n~ performs~ the
~following~
operations:\\
~~~~-For ~all~ m \thicksim n, do\\
~~~~~~~~y_{\{n,m\}}^{k+1}(n)=\frac{y_{\{n,m\}}^{k}(n)-y_{\{n,m\}}^{k}(m)}{2}+\frac{x^{k}_{n}-x^{k}_{m}}{2},\\
~~~~-x^{k+1}_{n}=prox_{\tau_{k}g_{n}/d_{n}}[(1-\tau_{k}\mu_{k}^{-1})x^{k}_{n}-\frac{\tau_{k}}{d_{n}}\nabla
f_{n}(x^{k}_{n})\\
~~~~~~~~~~~~~~~~~+\frac{\tau_{k}}{d_{n}}\sum_{m:\{n\thicksim m\}\in
E}(\mu_{k}^{-1}x^{k}_{m}+y^{k}_{\{n,m\}}(m))].\\
~~~~-For ~all~ m \thicksim n,~ send
\{x^{k+1}_{n},y_{\{n,m\}}^{k+1}(n)\}~to~ Neighbor~
m.\\
\bullet ~~For ~any~ agent ~n\neq\mathcal{B}, ~
x^{k+1}_{n}=x^{k}_{n},
~and~~y_{\{n,m\}}^{k+1}(n)=y_{\{n,m\}}^{k}(n)\\
~~~~for~ all~ m \thicksim n.\\
 \bullet ~~Increment~ k.
\end{array}
$$

\end{algorithmic}
\end{algorithm}

\begin{ass}
The collections of sets $\{\mathcal{B}_{1},\mathcal{B}_{2},\ldots\}$
such that $\mathbb{P}[\zeta^{1} = \mathcal{B}_{i}]$ is positive
satisfies $ \bigcup\mathcal{B}_{i} =Q$.
\end{ass}

\begin{thm}
  Assume that the minimization Problem (5.1) is consistent,
$\liminf_{k\rightarrow\infty}\mu_{k}>0$, and
$\liminf_{k\rightarrow\infty}\tau_{k}>0$.  Let Assumption 5.1,
Assumption 6.1 and 6.2 hold true, and
$\frac{1}{\liminf_{k\rightarrow\infty}\tau_{k}}-\frac{1}{\liminf_{k\rightarrow\infty}\mu_{k}}>\frac{L}{2}$.
Let $(x^{k}_{n})_{n\in Q}$ be the sequence generated by DASPDSDS for
any initial point $( x^{0}, y^{0})$ . Then the sequence
$x^{k}_{1},\ldots,x^{k}_{|Q|}$ converges to a solution of Problem
(5.1).
\end{thm}

\begin{proof}
Let  $(\bar{f}, \bar{g}, h)=(f\circ D^{-1}, g\circ D^{-1}, h)$ where
$f, g, h$  and $D$  are the ones defined in the Problem 6.2. By
Equations (3.24a). We write these equations more compactly as
$(y^{k+1}, x^{k+1}) = T^{k}(y^{k}, x^{k})$ , the operator $T^{k}$
acts in the space $\mathcal{Z}=\mathcal {Y}\times\mathcal{R}$, and
$\mathcal{R}$ is the image of $\mathcal{X}^{|Q|}$ by $D$. then from
the proof of Theorem 3.2, we know $T^{k}$ is  $\eta_{k}$-averaged
operator. Defining the selection operator $\mathcal{S}_{n }$ on
$\mathcal{Z}$ as $\mathcal{S}_{n }(y, Dx) =
(y_{\epsilon}(n)_{\epsilon\in Q:n\in\epsilon}, x_{n})$. So, we
obtain that $\mathcal{Z} = \mathcal{S}_{1 }(\mathcal{Z})\times
\cdots\times \mathcal{S}_{|Q| }(\mathcal{Z})$ up to an element
reordering. Identifying the set $\mathcal{J}$ introduced in the
notations of Section 4.1 with $Q$, the operator $T^{(\zeta^{k})}$ is
defined as follows:
$$
\mathcal{S}_{n }(T^{k(\zeta^{k})}(y, Dx)) = \left\{
\begin{array}{l}
\mathcal{S}_{n }(T^{k}(y, Dx)),\,\,\, \,\,if \,n\,
\in\zeta^{k},\\
\mathcal{S}_{n }(y, Dx),\,\,\,\,\,\,\,\,\,\,\, \,\,if \,n\,
\neq\zeta^{k} .
\end{array}
\right.
$$
Then by Theorem 4.1, we know the sequence $(y^{k+1}, Dx^{k+1}) =
T^{k,(\zeta^{k+1})}(y^{k}, Dx^{k})$ converges almost surely to a
solution of Problem (3.25). Moreover, from Lemma 6.1, we have the
sequence $x^{k}$ converges almost surely to a solution of Problem
(5.1).\\
Therefore we need to show that the operator $T^{k,(\zeta^{k+1})}$ is
translated into the DASPDSDS algorithm. The definition (6.1) of  $h$
shows that
$$h^{\ast}(\varphi)=\Sigma_{\epsilon\in E}\iota_{\mathcal
{C}_{2}^{\perp}}(\varphi_{\epsilon}),$$ where
$\mathcal{C}_{2}^{\perp}= \{(x,-x) : x\in\mathcal {X}\}$. Therefore,
writing
$$(\varsigma^{k+1}, \upsilon^{k+1}=Dq^{k+1}) = T^{k}(y^{k},
\lambda^{k}=Dx^{k}),$$  then by Eq. (3.24a),
$$\varsigma_{\epsilon}^{k+1}=proj_{\mathcal{C}_{2}^{\perp}}(y^{k}_{\epsilon}+\mu_{k}^{-1}\lambda^{k}_{\epsilon}).$$
Observe that contrary to the case of the synchronous algorithm
(6.3), there is no reason here for which
$proj_{\mathcal{C}_{2}^{\perp}}(y^{k}_{\epsilon})= 0$.  Getting back
to $(y^{k+1},D x^{k+1}) = T^{k,(\zeta^{k+1})}(y^{k},
\lambda^{k}=Dx^{k})$, we have for all $n \in \zeta^{k+1}$ and all $m
\thicksim n$,
\begin{align*}
y^{k+1}_{\{n,m\}}(n) &=
\frac{y^{k+1}_{\{n,m\}}(n)-y^{k+1}_{\{n,m\}}(m)}{2}+\frac{\lambda^{k+1}_{\{n,m\}}(n)-\lambda^{k+1}_{\{n,m\}}(m)}{2}\\
&=\frac{y^{k+1}_{\{n,m\}}(n)-y^{k+1}_{\{n,m\}}(m)}{2}+\frac{x^{k}_{n}-x^{k}_{m}}{2}.
\end{align*}
By Equation (3.24b) we also get
$$\upsilon^{k+1}=\arg\min_{w\in\mathcal
{R}}[\bar{g}(w)+\langle\nabla
\bar{f}(y^{k}),w\rangle+\frac{\|w-\lambda^{k}+\tau_{k}(2y^{k+1}-y^{k})\|^{2}}{2\tau_{k}}].$$
Upon noting that $\bar{g}(Dx) = g(x)$ and $\langle\nabla
\bar{f}(\lambda^{k}),Dx\rangle= \langle(D^{-1})^{\ast}\nabla
f(D^{-1}Dx^{k}),Dx\rangle = \langle\nabla f(x^{k}), x\rangle$, the
above equation becomes
$$q^{k+1}_{n}=\arg\min_{w\in\mathcal
{X}}[g(w)+\langle\nabla
f(x^{k}),w\rangle+\frac{\|D(w-x^{k})+\tau_{k}(2y^{k+1}-y^{k})\|^{2}}{2\tau_{k}}].$$
Recall that $(D^{\ast}Dx)n = d_{n}x_{n}$. Hence, for all $n\in
\zeta^{k+1}$, we get after some computations
$$x^{k+1}_{n}=prox_{\tau_{k}g_{n}/d_{n}}[x^{k}_{n}-\frac{\tau_{k}}{d_{n}}\nabla
f_{n}(x^{k}_{n})+\frac{\tau_{k}}{d_{n}}(D^{\ast}(2y^{k+1}-y^{k}))_{n}].$$
 Using the identity $(D^{\ast}y)_{n}=\sum_{m:\{n,m\}\in
E}y_{\{n,m\}}(n)$ , it can easy check these equations coincides with
the $x$-update  in the DASPDSDS algorithm.

\end{proof}

\section{Numerical experiments}
In this section, we present some numerical experiments to verify the effective of our proposed
iterative algorithms. All experiments were performed in MATLAB (R2013a) on Lenovo laptop with Intel
(R) Core(TM) i7-4712MQ 2.3GHz and 4GB memory on the windows 7 professional operating system.

We consider the following $l_{1}$-regularization problem,
\begin{equation}\label{l1-regular}
\min_{x\in R^{n}}\ \frac{1}{2}\|Ax-b\|_{2}^{2} + \lambda \|x\|_{1},
\end{equation}
where $\lambda>0$ is the regularization parameter, the system matrix $A\in R^{m\times n}$, $b\in R^{m}$ and
$x\in R^{n}$. Let $\{W_{i}\}_{i=1}^{N}$ be a partition of $\{1,2, \cdots,m\}$, the optimization problem (\ref{l1-regular})
then writes,
\begin{equation}\label{l1-regu-path}
\min_{x\in R^{n}}\ \sum_{k=1}^{N}\sum_{i\in W_k} \frac{1}{2}\|A_i x - b_i\|_{2}^{2} + \lambda \|x\|_{1}.
\end{equation}
Further, splitting the problem (\ref{l1-regu-path}) between the batches, we have
\begin{equation}
\min_{x\in R^{Nn}}\ \sum_{k=1}^{N}\left (  \sum_{i\in W_k} \frac{1}{2}\|A_i x - b_i\|_{2}^{2} + \frac{\lambda}{N}\|x\|_1 \right)
+ \iota_{C}(x),
\end{equation}
where $x= (x_1, x_2, \cdots, x_N)$ is in $R^{Nn}$.

We first describe how  the system matrix $A$ and a $K$-sparse signals $x$ were generated.
Let the sample size $m=1/4 n$ and  $K= 1/64 n$. The system matrix $A$ is random generated from Gaussian
distribution with $0$ mean and $1$ variance. The $K$-sparse signal $x$ is generated by random perturbation with $K$
values nonzero which are obtained with uniform distribution in $[-2,2]$ and the rest are kept with zero. Consequently,
the observation vector $b = Ax + \delta$, where $\delta$ is added Gaussian noise with $0$ mean and $0.05$ standard variance.
Our goal is to recover the sparse signal $x$ from the observation vectors $b$.

To measure the performance of the proposed algorithms, we use $\ell_2$-norm error between the reconstructed
variable $x_{rec}$ and the true variable $x_{true}$, function values ($fval$) and iteration numbers ($k$). That is,
$$
Err = \|x_{rec}-x_{true}\|_2, \quad fval = \frac{1}{2}\|Ax_{rec}-b\|_{2}^{2}.
$$
We set the stopping criteria as
$$
\frac{\|x^{k+1}-x^{k}\|_{2}}{\|x^{k}\|_{2}} < \epsilon,
$$
where $\epsilon$ is a given small constant; Otherwise, the maximum iteration numbers $40000$ reached.

\begin{table}[htbp]
\footnotesize
\centering
\caption{Numerical results obtained by Algorithm \ref{algorithm4}}
\begin{tabular}{c|c|cccccccccccc}
\hline
Problem & Block & & \multicolumn{3}{c}{$\epsilon = 10^{-5}$}  & & \multicolumn{3}{c}{$\epsilon = 10^{-6}$} & &
\multicolumn{3}{c}{$\epsilon = 10^{-8}$} \\ \cline{4-6} \cline{8-10} \cline{12-14}
size & size & & $Err$ & $fval$ & $k$ & & $Err$ & $fval$ & $k$ & & $Err$ & $fval$ & $k$\\
\hline
\multirow{2}[1]{*}{$n=10240$} & $N=2$ & & $0.0616$ & $0.2970$ &
$20343$ & & $0.0488$ & $0.2901$ & $20839$ & & $0.0479$
& $0.2890$ & $22777$ \\
& $N=4$ & & $0.0949$ & $0.2869$ & $40285$ & & $0.0497$ & $0.2911$ & $41364$ & & $0.0480$
& $0.2890$ & $44771$ \\
\hline
\multirow{2}[1]{*}{$n=20480$} & $N=4$ & & $0.8746$ & $0.2710$ &
$68008$ & & $0.0511$ & $0.3063$ & $73661$ & & $0.0477$
& $0.3046$ & $78879$ \\
& $N=6$ & & $$ & $$ &
$$ & & $$ & $$ & $$ & & $$
& $$ & $$ \\

\hline

\end{tabular}
\end{table}

\begin{table}[htbp]
\footnotesize
\centering
\caption{Numerical results obtained by Algorithm \ref{algorithm5}}
\begin{tabular}{c|c|cccccccccccc}
\hline
Problem & Block & & \multicolumn{3}{c}{$\epsilon = 10^{-5}$}  & & \multicolumn{3}{c}{$\epsilon = 10^{-6}$} & &
\multicolumn{3}{c}{$\epsilon = 10^{-8}$} \\ \cline{4-6} \cline{8-10} \cline{12-14}
size & size & & $Err$ & $fval$ & $k$ & & $Err$ & $fval$ & $k$ & & $Err$ & $fval$ & $k$\\
\hline
\multirow{2}[1]{*}{$n=10240$} & $N=2$ & & $0.0469$ & $0.4599$ &
$-$ & & $0.0469$ & $0.3101$ & $-$ & & $0.0475$
& $0.2596$ & $-$ \\
& $N=4$ & & $0.0460$ & $0.5515$ & $-$ & & $0.0465$ & $0.3479$ & $-$ & & $0.0465$
& $0.4180$ & $-$ \\
\hline

\end{tabular}
\end{table}

\noindent \textbf{Acknowledgements}

This work was supported by the National Natural Science Foundation
of China (11131006, 41390450, 91330204, 11401293), the National
Basic Research Program of China (2013CB 329404), the Natural Science
Foundations of Jiangxi Province (CA20110\\
7114, 20114BAB 201004).


\begin{thebibliography}{00}

\bibitem{[1]}L. Condat, ¡°A primal-dual splitting method for convex optimization
involving Lipschitzian, proximable and linear composite terms,¡±
Journal of Optimization Theory and Applications, vol. 158, no. 2,
pp. 460¨C479, 2013.
\bibitem{[2]} Bianchi P, Hachem W  and Iutzeler F 2014 A Stochastic coordinate descent
primal-dual algorithm and applications to large-scale composite
(arXiv:1407.0898v1 [math.OC] 3 Jul 2014) Optimization
\bibitem{[3]}Bauschke, H.H., Combettes, P.L.: Convex Analysis and Monotone Operator Theory in Hilbert Spaces.
Springer, New York (2011)
\bibitem{[4]}Ogura, N., Yamada, I.: Non-strictly convex minimization over the fixed point set of an asymptotically
shrinking nonexpansive mapping. Numer. Funct. Anal. Optim. 23(1¨C2), 113-137 (2002)
\bibitem{[5]} Byrne, C.: A unified treatment of some iterative algorithms in signal processing and image reconstruction.
Inverse Probl. 20, 103-120 (2004)
\bibitem{[6]} Combettes, P.L.: Solving monotone inclusions via compositions of nonexpansive averaged operators.
Optimization 53, 475-504 (2004)
\bibitem{[7]} Geobel, K., Kirk, W.A.: Topics in Metric Fixed Point Theory. Cambridge Studies in Advanced Mathematics,
vol. 28. Cambridge University Press, Cambridge (1990)
\bibitem{[8]}Bruck R E and Passty G B 1979 Almost convergence of the infinite product of resolvents in Banach spaces Nonlinear Anal. 3 279-282.
\bibitem{[9]}Bruck R E and Reich S 1977 Nonexpansive projections and resolvents in Banach spaces Houston J. Math. 3  459-470.
\bibitem{[10]}R. T. Rockafellar, Convex analysis, Princeton Mathematical Series, No.
28. Princeton University Press, Princeton, N.J., 1970.
\bibitem{[11]} I. Daubechies, M. Defrise, and C. De Mol, ¡°An iterative
thresholding algorithm for linear inverse problems with a sparsity
constraint,¡± Communications on pure and applied mathematics, vol.
57, no. 11, pp. 1413-1457, 2004.
\bibitem{[12]} R. Tibshirani, ¡°Regression shrinkage and selection via the lasso,¡±
Journal of the Royal Statistical Society. Series B (Methodological),
pp. 267-288, 1996.
\bibitem{[13]} Forero, P A, Cano A and Giannakis G B 2010 Consensus-based
distributed support vector machines  The Journal of Machine Learning
Research  99 1663-1707.
\bibitem{[14]} Agarwal A, Chapelle O, Dud\'{\i}k M, and  Langford J 2011 A reliable effective
terascale linear learning system arXiv preprint arXiv:1110.4198.
\bibitem{[15]}P. Bianchi and J. Jakubowicz, ¡°Convergence of a multi-agent projected
stochastic gradient algorithm for non-convex optimization,¡± IEEE
Transactions on Automatic Control, vol. 58, no. 2, pp. 391- 405,
February 2013.
\bibitem{[16]} S.S. Ram, V.V. Veeravalli, and A. Nedic, ¡°Distributed and recursive
parameter estimation in parametrized linear state-space models,¡±
IEEE Trans. on Automatic Control, vol. 55, no. 2, pp. 488-492, 2010.
\bibitem{[17]} P. Bianchi, G. Fort, and W. Hachem, ¡°Performance of a distributed
\bibitem{[18]} Yu. Nesterov, ¡°Efficiency of coordinate descent methods on huge-scale
optimization problems,¡± SIAM Journal on Optimization, vol. 22, no.
2, pp. 341-362, 2012.
\bibitem{[19]} O. Fercoq and P. Richt¡äarik, ¡°Accelerated, parallel and proximal coordinate
descent,¡± arXiv preprint arXiv:1312.5799, 2013.
\bibitem{[20]} M. Ba¡¦c¡äak, ¡°The proximal point algorithm in metric spaces,¡± Israel
Journal of Mathematics, vol. 194, no. 2, pp. 689-701, 2013.



\end{thebibliography}
\end{document}